\documentclass[11pt]{imsart}
\usepackage{fullpage}

\usepackage{amsmath,amsthm}
\usepackage{amssymb}
\RequirePackage[round]{natbib}
\RequirePackage[colorlinks,citecolor=blue,urlcolor=blue, linkcolor=blue]{hyperref}
\RequirePackage{graphicx}

\usepackage[capitalise]{cleveref} 
\crefformat{section}{\S#2#1#3}
\crefmultiformat{section}{\S\S#2#1#3}{ and~#2#1#3}{, #2#1#3}{ and~#2#1#3}
\crefname{equation}{}{}
\crefname{Equation}{equation}{equations}

\theoremstyle{plain}
\newtheorem{theorem}{Theorem}
\newtheorem{proposition}{Proposition}
\newtheorem{lemma}{Lemma}
\newtheorem{corollary}{Corollary}

\theoremstyle{definition}

\newtheorem{definition}{Definition}
\newtheorem{remark}{Remark}
\newtheorem{example}{Example}

\newcommand{\appendixnumbering}{\setcounter{theorem}{0}\setcounter{lemma}{0}\setcounter{example}{0}\setcounter{property}{0}\setcounter{assumption}{0}\setcounter{proposition}{0}\setcounter{corollary}{0}\setcounter{definition}{0}\setcounter{figure}{0}\setcounter{table}{0}\renewcommand{\thetheorem}{S\arabic{theorem}}\renewcommand{\thelemma}{S\arabic{lemma}}\renewcommand{\theexample}{S\arabic{example}}\renewcommand{\theproperty}{S\arabic{property}}\renewcommand{\theassumption}{S\arabic{assumption}}\renewcommand{\theproposition}{S\arabic{proposition}}\renewcommand{\thecorollary}{S\arabic{corollary}}\renewcommand{\thedefinition}{S\arabic{definition}}\renewcommand{\thefigure}{S\arabic{figure}}\renewcommand{\thetable}{S\arabic{table}}}

\usepackage{float}
\usepackage{url}
\graphicspath{{./art/}}

\usepackage[plain,noend]{algorithm2e}

\makeatletter
\renewcommand{\algocf@captiontext}[2]{#1\algocf@typo. \AlCapFnt{}#2} \def\@algocf@capt@plain{top}
\renewcommand{\algocf@makecaption}[2]{\addtolength{\hsize}{\algomargin}\sbox\@tempboxa{\algocf@captiontext{#1}{#2}}\ifdim\wd\@tempboxa >\hsize \hskip .5\algomargin \parbox[t]{\hsize}{\algocf@captiontext{#1}{#2}}\else \global\@minipagefalse \hbox to\hsize{\box\@tempboxa}\fi \addtolength{\hsize}{-\algomargin}}
\makeatother

\def\T{\intercal}

\def\E{{\rm E}}

\def\R{\mathbb R}
\def\M{\mathbb M}
\def\N{\mathbb N}

\def\RV{\mathrm{RV}}

\let \epsilon \varepsilon

\RequirePackage{amsmath,amsfonts}

\usepackage{subcaption} \usepackage{tikz}       

\numberwithin{equation}{section}

\newcommand{\Pb}{{\rm pr}} 

\renewcommand{\P}{\Pb} 

\newcommand{\rbrac}[1]{\left(#1\right)}
\newcommand{\sqbrac}[1]{\left[#1\right]}
\newcommand{\cbrac}[1]{\left\{#1\right\}}
\newcommand{\abs}[1]{\left\lvert #1 \right\rvert}
\newcommand{\norm}[1]{\left\lVert #1 \right\rVert}

\newcommand{\ovee}{\tikz[baseline=-0.5ex] \node[draw,circle,inner sep=0.1ex]{\tiny$\vee$};}

\newcommand{\commb}[1]{{#1}}

\begin{document}

\begin{frontmatter}
\title{On the Universal Calibration of Heavy-tailed Combination Tests}
\runtitle{Universal Calibration of Heavy-tailed Combination Tests}

\begin{aug}
\author[A]{\fnms{Parijat}~\snm{Chakraborty}\ead[label=e1]{cparijat@umich.edu}},
\author[A]{\fnms{F.~Richard}~\snm{Guo}\ead[label=e2]{ricguo@umich.edu}},
\author[A]{\fnms{Kerby}~\snm{Shedden}\ead[label=e3]{kshedden@umich.edu}},
\and
\author[A]{\fnms{Stilian}~\snm{Stoev}\ead[label=e4]{sstoev@umich.edu}}

\runauthor{P.~Chakraborty, F.~R. Guo, K.~Shedden and S.~Stoev}

\address[A]{Department of Statistics, University of Michigan\printead[presep={,\ }]{e1,e2,e3,e4}}
\end{aug}

\maketitle

\begin{abstract}
It is often of interest to test a global null hypothesis using multiple, possibly dependent $p$-values by combining their strengths while controlling the type-I error.
Recently, several heavy-tailed combination tests, such as the harmonic mean test and the Cauchy combination test, have been proposed: 
they transform $p$-values into heavy-tailed random variables before combining them into a single
test statistic.  
The resulting tests, which are calibrated under some form of independence assumption 
among the $p$-values, have been shown to be rather robust to dependence asymptotically as the $\alpha$ level gets small. 
Yet, it has remained an open problem to understand this general phenomenon and characterize how such tests behave under dependence.
Using the framework of multivariate regular variation from extreme value theory, we show that for a class of combination tests that are homogeneous, the asymptotic level of the test can be expressed using the angular measure under multivariate regular variation. 
This measure characterizes the dependence of the transformed heavy-tailed variables in their upper tails, or equivalently, the dependence of the $p$-values near zero. 
We use this result to study several tests. 
The harmonic mean test, which coincides with the Pareto linear combination test, is shown to be universally calibrated regardless of the tail dependence; further, this test is shown to be the only one that achieves universal calibration among all homogeneous heavy-tailed combination tests. 
In contrast, the Cauchy combination test is shown to be universally honest but often conservative; the Dunn--\v{S}id{\'a}k correction, also known as Tippett's method, while being honest, is calibrated if and only if the underlying $p$-values are independent near zero.
These theoretical findings are corroborated with simulations and an application to independence testing with survey data. 
\end{abstract}


\end{frontmatter}

\section{Introduction}
It is often of interest to test a global null hypothesis using multiple $p$-values, each of which is marginally uniformly distributed on the unit interval if the global null holds. 
Examples abound, including set-based analysis in GWAS \citep{wu2010powerful}, rare-variant analysis in genetics \citep{liu2019acat}, meta-analysis \citep{singh2005combining}, variable and model selection \citep{meinshausen2010stability}, derandomizing data splitting \citep{guo2025rank}, to name a few. 
Depending on the construction of these $p$-values, they are often, though not always, correlated and their dependence structure is typically unknown. 
In this paper, we focus on the setting where the raw data for constructing these $p$-values are unavailable and we must treat the $p$-values themselves as the summary of all the evidence we have against the global null hypothesis. 
Though beyond the scope of this paper, it is worth mentioning that the raw data, when available, can be used to estimate the dependence structure to improve power \citep{guo2025rank}. 

In the above setting, it is natural to consider a \emph{combination test} that outputs a single $p$-value by combining the strengths from multiple $p$-values, an idea that dates back to the early works of \citet{tippett1931methods}, \citet{fisher1948combining}, \citet{good1958significance}, \citet{lancaster1961combination} and \citet{simes1986improved}. Ideally, the combined $p$-value has more power against the global null than any of the original $p$-values. While the early works in this area often assume independence of the $p$-values, the more recent development has shifted towards methods that can control the family-wise type-I error, at least approximately, under a wide variety of dependence among the $p$-values; see, for example, \citet{meng1994posterior,wilson2019harmonic,liu:xie:2020,vovk2020combining,diciccio2020exact} and \citet{vovk2021values}. 

Among the most notable recent developments are the heavy-tailed combination tests, which combine multiple, possibly dependent $p$-values after transforming them to heavy-tailed random variables such as Pareto or Cauchy. 
In particular, \citet{wilson2019harmonic} proposed the harmonic mean combination test, which dates back to \citet{good1958significance};
and \citet{liu:xie:2020} developed the Cauchy combination test, which is widely used in genomics \citep{liu2019acat,reay2021advancing}. 

The idea and setup behind both of these tests are very similar: let $P_1, \dots, P_d$ be the $p$-values associated with $d$ tests, which are distributed according to Uniform$(0,1)$ under the global null hypothesis 
$\mathcal{H}_0$. 
In the context where each $P_i$ is constructed to test a corresponding hypothesis $\mathcal{H}_{0,i}$, the global null is taken to be $\mathcal{H}_0 := \bigcap_{i=1}^d \mathcal{H}_{0,i}$.
Throughout the paper, we say a distribution function $F$ has a \emph{heavy (upper) tail} if 
\[ 1-F(x) \sim L(x) x^{-\beta}, \quad  x \to +\infty \]
for a \emph{tail exponent} or \emph{tail index} $\beta > 0$ and a slowly varying function $L$. 
The function $L$ is said to be slowly varying (at infinity) if $L(tx)/L(t)\to 1$ as $t\to\infty$ for every $x>0$; see, e.g., \citet[p.~13]{resnick:1987}. 
The transformed random variables are given by 
\begin{equation} \label{eq:X}
X_i:= F^{-1}(1-P_i), \quad i =1,\dots,d,
\end{equation}
so that a small value of $P_i$ is mapped to the upper tail of $X_i$. In this case we write $X_i\in \RV_{\beta}$ and $\bar{F}=1-F\in \RV_{\beta}$. 
Then, one considers the {\em linear} combination test statistic:
\[T_{F,w} := \sum_{i=1}^d w_i X_i, \ \ \mbox{ for some }  w_i\ge 0 \text{ s.t. } \sum_{i=1}^d w_i = 1. \]
\commb{The weights above can be chosen according to the nature of tests that produce the $p$-values.
If no prior knowledge about these tests is available, we typically use equal weights.
For the rest of the paper, the results about $T_{F,w}$ apply to any weights, as long as the procedure for choosing them is independent of the $p$-values.}

For a prespecified level $\alpha \in (0,1)$, the global null $\mathcal{H}_0$ is rejected when $T_{F,w}$ exceeds a corresponding critical value $\tau_{\alpha}$. 
Typically, $\tau_{\alpha}$ is set to be $F^{-1}(1-\alpha)$, the upper $\alpha$ quantile of $F$. 
For a pre-specified level $\alpha \in (0,1)$, we say the combination test is  \emph{calibrated} if $\Pb_0[ T_{F,w}>\tau_{\alpha}] = \alpha$, whereas we say the test is \emph{honest} if $\Pb_0[ T_{F,w}>\tau_{\alpha}] \leq \alpha$. 
Here, $\Pb_0$ means the probability holds with respect to any \emph{fixed} data-generating distribution under $\mathcal{H}_0$. 
It is worth mentioning that, if $T_{F,w}$ is calibrated but one or more $p$-values supplied can be conservative, i.e., following a super-uniform distribution under $\mathcal{H}_0$, then the test is still honest because $T_{F,w}$ is non-increasing in $P_1, \dots, P_d$. 
When a final $p$-value is also desired, the combined $p$-value is given by $P_{F,w} := 1 - F(T_{F,w})$. 
 
Taking $F$ to be the standard Pareto distribution with index 1, namely $F(x) = 1-1/x$ for $x >1$, recovers the weighted harmonic mean $p$-value \citep{wilson2019harmonic,good1958significance}.  
Taking $F$ to be the standard Cauchy distribution, namely $F(x) = \pi^{-1} \arctan x + 1/2$ for $x \in \mathbb{R}$, leads to the Cauchy combination test  \citep{liu:xie:2020}.
The Cauchy combination test is calibrated under two extreme dependencies: when the $p$-values are independent or perfectly positively correlated, we have
\[ T_{F,w} \stackrel{d}{=} \left(\sum_{i=1}^d w_i \right)\cdot X_1 = X_1; \]
see also \cref{ex:SaS} in the Supplementary Material.  
Moreover, several theoretical and simulation studies have found that this calibration is robust to certain non-trivial dependence in the $p$-values. 
For example, it is established that when every pair of the $p$-values follow a normal copula \citep{liu:xie:2020} or several other copulas \citep{long2023cauchy}, the Cauchy combination test is asymptotically calibrated, as made precise in the following definition.

\begin{definition}[asymptotic calibration and honesty] \label{def:plain}
Given critical values $\tau_{\alpha}$, the combination test $T$ is said to be asymptotically 
\begin{equation*}
\begin{cases} 
\text{calibrated}, \quad & \text{if } \lim_{\alpha\downarrow 0} \alpha^{-1} \Pb_0( T >\tau_\alpha) =1; \\
\text{honest}, \quad & \text{if }  \limsup_{\alpha\downarrow 0} \alpha^{-1} \Pb_0( T >\tau_\alpha) \leq 1; \\
\text{conservative}, \quad & \text{if }  \limsup_{\alpha\downarrow 0} \alpha^{-1} \Pb_0( T >\tau_\alpha) < 1. \end{cases}
\end{equation*}
\end{definition}

In many applications, small levels of $\alpha$ are of interest and the above asymptotic notions of calibration and honesty are useful for approximately controlling the type-I error. 
Hence, for the rest of the paper, unless stated otherwise, we will simply take calibration and honesty to mean asymptotic calibration and asymptotic honesty, respectively. 

\medskip In this line of work, the foremost question is to identify a family of dependence structures that is as large as possible to plausibly accommodate practical settings, under which the heavy-tailed combination tests remain asymptotically calibrated or honest. 
The earlier results can be generalized to the assumption that $X_1, \dots, X_d$ are pairwise asymptotically independent in their upper tails, defined as follows. 

\begin{definition}[upper tail dependence coefficient and asymptotic independence]
For random variables $X_1, X_2$ with a common distribution function $F$, their (upper tail) dependence coefficient is 
\begin{equation} \label{e:lambda-def-intro}
\lambda(X_1,X_2) := \lim_{p\uparrow 1} \Pb\{ F(X_1) > p | F(X_2)> p \},
\end{equation}
whenever the limit exists. When $\lambda(X_1,X_2)=0$, we say that $X_1,X_2$ are asymptotically (upper tail) independent; otherwise, they are asymptotically (upper tail) dependent.
\end{definition}

In light of \cref{eq:X}, the dependence coefficient between $X_i$ and $X_j$ equals the bivariate lower-tail dependence coefficient of 
the copula between $p$-values $P_i$ and $P_j$; see also \citet{joe2015dependence}. 
A well-known result due to \citet{sibuya:1960} shows that random variables that follow any non-degenerate bivariate normal copula are asymptotically independent. 
In fact, as observed in the recent work of \citet{fang2023heavy} and \citet{gui2025aggregating}, the asymptotic calibration of the Cauchy combination test can be established under the assumption of pairwise asymptotic independence of $X_1, \dots, X_d$, which is weaker than assuming a certain copula underlying every pair of $p$-values.

Naturally, this leads to the question \commb{of} whether a heavy-tailed combination test remains calibrated or honest when $X_1, \dots, X_d$ are pairwise asymptotically \commb{\emph{dependent}}, which arises in many statistical contexts; see \cref{sec:examples}.
\commb{In this work, we address this question using a general framework for multivariate dependence called \emph{multivariate regular variation}. This framework, while well-developed in the literature on extreme value theory 
\citep[see, e.g.,][]{resnick:1987,beirlant2004statistics,dehaan:ferreira:2006,resnick:2007,kulik:soulier:2020,mikosch:wintenberger:2024,resnick:2024}, is perhaps  
lesser-known in the broader statistical literature. The core technical tools can also be traced to the works of \citet{barbe:fougeres:genest:2006} and \citet{embrechts:lambrigger:wuthrich:2009} in the context of quantifying {\em extreme Value-at-Risk}; see also \citet{yuen:stoev:cooley:2020}. One of the biggest advantages of adopting this well-developed framework in the context of combination tests is that it allows $X_1, \dots, X_d$ to be dependent in their tails, or equivalently, the $p$-values $P_1, \dots, P_d$ to be dependent near zero --- a regime unexplored in the existing literature.}

\commb{The concurrent and independent work of \citet{gui:mao:wang:wang:2025} provides a theoretical analysis for both calibration and power of heavy-tailed combination tests within the same framework of multivariate regular variation. Our approach is different: we focus on theoretically characterizing the calibration of \emph{any} homogeneous, heavy-tailed combination test and use simulations to study power. Our main result, \cref{t:unique_univ_calibrated}, shows that the Pareto linear combination test is the \emph{only} homogeneous test that is universally calibrated under all multivariate regular variation dependence structures. We also completely characterize the calibration properties of the Cauchy combination test under the regular variation framework.}

\commb{The theoretical results in this paper on the asymptotic calibration of heavy-tailed combination tests are for a fixed number of hypotheses $d$. In practice, however, combination tests are often applied in the high-dimensional regime. While the theoretical analysis of the Pareto combination test in the high-dimensional regime is beyond the scope of this paper, in \cref{subsec: pct in high dim}, motivated by a question of an anonymous referee, we examined its behavior using simulations. Our results show that the Pareto combination test has favorable performance and essentially matches the power of the max test, shown to be optimal in the high-dimensional setting of \citet{li2020optimality}.}
\section{Multivariate regular variation and asymptotic calibration of combination tests}  \label{sec:MRV}
\subsection{Multivariate regular variation} \label{sec:MRV-def}

In this section, we review the fundamental notion of multivariate regular variation and describe how it provides a natural framework for quantifying the asymptotic calibration of combination tests.
The reader is referred to \cref{sec:supp-MRV} of the Supplementary Material for a brief introduction to multivariate regular variation. 
\begin{definition} A random vector $X=(X_i)_{i=1}^d$ is multivariate regularly varying 
if there exists a positive function $b(t)\to \infty$, and a non-zero Borel measure $\mu$ on 
$\R^d\setminus\{0\}$ such that 
\begin{equation}\label{e:d:RV}
b(t) \,\Pb\left( X \in t\cdot A \right) \longrightarrow \mu(A) \quad \text{as } t \to \infty
\end{equation}
for all Borel sets $A\subset \R^d\setminus\{0\}$ that are bounded away from $0$ and
$\mu(\partial A) =0$, 
where $\partial A$ is the boundary of $A$.
In this case, we write $ X\in \RV(\R^d, b(\cdot),\mu).$
\end{definition}

The measure $\mu$, which need not be a probability measure, is referred to as the {\em exponent measure} of $X$. 
It characterizes the asymptotic behavior of the {\em extremes} of $X$, and in particular, the asymptotic (in)dependence property of the components of the vector $X$.  
For simplicity, assume that the
vector $X$ is \commb{standardized to have Pareto marginals}, i.e., $\commb{\Pb[X_i>t] = 1/t,\mbox{ for all }t\ge 1}$.
Let $F^{-1}(p) = \inf\{x: F(x) \geq p\}$ denote the inverse of a distribution function $F$. 
Then the (upper) tail-dependence coefficient between $X_i$ and $X_j$ is given by
\begin{align*}
\lambda(X_i,X_j) &= \lim_{p\uparrow 1} \Pb\{ X_i > F_{X_i}^{-1}(p) \, |\ X_j > F_{X_j}^{-1}(p) \}\\
& = \lim_{t\to\infty} t\, \Pb( X_i >t, X_j >t) = \lim_{t\to\infty} t \,\Pb( X/t \in A_i \cap A_j) = \mu(A_i\cap A_j),
\end{align*}
where $A_i = \{ x\, :\, x_i >1 \}$. Thus $\mu$ is fundamentally related to $\lambda(X_i,X_j)$, a quantity which characterizes the occurrence of joint positive extremes of $X_i$ and $X_j$. For example, if $\mu(A_i\cap A_j)=0$, then $\lambda(X_i,X_j)=0,$ implying that the extremes do not occur simultaneously, and therefore $X_i$ and $X_j$ are correctly said to be \emph{asymptotically (upper tail) independent}.

\begin{remark} \label{r:ai}
As noted in \commb{Remark 3.1} in \citet{gui:mao:wang:wang:2025}, it is well-known in the extreme value literature that, 
for heavy-tailed random vectors, bivariate asymptotic independence implies their multivariate regular variation with the exponent measure concentrated on the axes. \commb{Consequently, the dependency among $p$-values assumed in the combination test literature may be cast in this
framework. The seminal paper by \citet{liu:xie:2020} on the Cauchy 
combination test assumes that the $p$-values arise from a pairwise Gaussian copula. Due to the above mentioned property and \citet{sibuya:1960} this copula is multivariate regularly varying with exponent measure $\mu$ concentrated on the axes.}

\commb{While pieces of this property date back 
to \citet{berman1961convergence} (see Theorem 2 and also \S 8.3.2 in \citealp{beirlant2004statistics}), we are unable to find a documented proof of this result in the literature.} For completeness, we offer a proof following \cref{th:asym indep & heavy tails implies mrv} in \cref{sec:supp-MRV} of the Supplementary Material.
\end{remark}
 
\commb{In the rest of this section, we present a fundamental result on the general structure of the exponent measure of a regularly varying random vector (\cref{thm:tail-index}) followed by a result introducing the angular, or spectral, decomposition of the exponent measure (\cref{thm:polar}). \Cref{thm:polar} then allows us to state a key technical lemma that helps to establish the asymptotic 
calibration properties of \emph{any} homogeneous combination test (\cref{l:homogeneous}). The proofs of \cref{thm:tail-index,thm:polar} can be found in many comprehensive expositions
in the literature; see, e.g., Theorem 3.1 in \citet{lindskog:resnick:roy:2014}. See also the monographs by \citet{resnick:1987,resnick:2007,resnick:2024}, a more 
recent treatment \citep[Theorem 2.1.3 of][]{kulik:soulier:2020} and the references therein. \Cref{l:homogeneous} is borrowed from \citet{jansen:neblung:stoev:2023}, which contains its detailed proof.}

\begin{theorem}[Tail index theorem] \label{thm:tail-index} Let $X = (X_i)_{i=1}^d$ be a random vector in $\R^d$.

\begin{enumerate}
    \item [{\em (i)}]  If  $X\in \RV(\R^d,b(\cdot),\mu),$ then: 
\begin{enumerate}
    \item There exists $\beta>0$, referred to as the {\em tail index} of $X$,
    such that $b(t) = \ell(t) t^{\beta}$, for some slowly varying function $\ell:(0,\infty)\to (0,\infty)$.
    \item The measure $\mu$ is $\beta$-homogeneous, i.e., for all $t>0$, and all Borel sets 
    $A$ in $\R^d$ that are bounded away from $0$, we have
    \begin{equation}\label{e:thm:tail-index-mu-scaling}
    \mu(t A) = t^{-\beta} \mu(A) <\infty.
    \end{equation}
    \item The tail index $\beta$ is unique in the sense that if it also holds that $X \in \RV(\R^d,c(\cdot),\nu)$ with $c(t) = \ell_c(t) t^{\kappa}$ for a slowly varying function $\ell_c$, then 
\[ \beta = \kappa,\ \ \frac{b(t)}{c(t)} \to a>0,\ \mbox{ and }\ a \mu(A) =  \nu(A).\]
\end{enumerate}

\item [{\em (ii)}] Conversely, for every non-zero Borel measure $\mu$ on $\R^d\setminus\{0\}$ that satisfies
\eqref{e:thm:tail-index-mu-scaling} for some $\beta>0$, there exists  a random vector $X\in \RV(\R^d,b(\cdot),\mu)$, with $b(t) = \ell(t) t^\beta$ for a slowly varying function $\ell$.
\end{enumerate}
\end{theorem}

Part (i)~c of the theorem allows us to write $X \in \RV_{\beta}(\R^d,b(\cdot),\mu)$ that signifies the tail index $\beta$.  
Further, Part (i)~b shows that the measure $\mu$ is, up to rescaling, also unique and independent of the choice of the sequence $b(\cdot)$.
While there are several equivalent formulations of regular variation, the next one in terms of polar coordinates will be useful to us. 

\begin{theorem}\label{thm:polar}
We have $X\in \RV_\beta(\R^d,b(\cdot),\mu)$ if and only if 
for some (and hence any) norm $\|\cdot\|$ in $\R^d$, the following
two conditions hold:

\begin{enumerate}
    \item For a slowly varying function $L$, it holds that 
    $$
\Pb\left( \|X\| >t \right) \sim L(t)t^{-\beta},\ \ t\to\infty.
$$
\item As $t\to +\infty$, we have
\begin{equation}\label{e:thm:polar-ii}
\frac{X}{\|X\|} \, \bigg\vert\, \{\|X\|>t\} \stackrel{d}{\longrightarrow} \Theta,
\end{equation}
where $\Theta$ is a random vector taking values in the unit sphere $S_{\|\cdot\|}:=\{x\in \R^d: \|x\|=1\}$.
\end{enumerate}
Moreover, by adopting the polar coordinates $\Psi:\R^d\setminus\{0\} \to S_{\|\cdot\|} \times (0,\infty)$
where $\Psi(x) := (r(x),\theta(x))$, with $r(x):= \|x\|$ and $\theta(x) := x/\|x\|$, we have 
\begin{equation} \label{eqs:polar}
\mu\circ \Psi^{-1}(dr,d\theta) = c_\mu \, \beta\, r^{-\beta-1} dr \sigma(d\theta), 
\end{equation}
where $c_{\mu}:= \mu(\{ r>1\})$ and $\sigma$ is the probability measure of $\Theta$ in \eqref{e:thm:polar-ii}. 
\end{theorem}

This result shows that the measure $\mu$, when viewed in polar coordinates, factors into 
the product of a radial power-law type component and an angular component. Essentially it tells us that radially $X$ 
behaves like a heavy-tailed random variable and \emph{when} $\|X\|$ is extreme, the distribution of the
directions $X/\|X\|$ is asymptotically governed by $\sigma$. As a result, $\sigma$ is called the \emph{angular probability measure} associated with $\mu$. 
By analogy with the theory on infinitely divisible laws, 
$\sigma$ is also referred to as the {\em spectral measure} of $\mu$.
The angular measure enables us to evaluate the tail probability of a homogeneous function of $X$, as given by the next result. 
A function $h: \R^d \to \R$ is \emph{$1$-positively-homogeneous} if $h(a x) = a h(x)$ holds for every $a > 0$. 
In what follows, we use $\R_+$ to denote the non-negative real line and $\R_+^{d}$ to denote the $d$-dimensional non-negative orthant. 

\begin{lemma}[see Proposition 2.5 in \citealp{jansen:neblung:stoev:2023}]
\label{l:homogeneous}
Let $X \in  \RV_{\beta}(\R^d,b(\cdot),\mu)$ and let $\sigma$ be the corresponding angular probability measure. 
For any continuous, $1$-positively-homogeneous function $h:\R^d\to \R_+$, we have
$$
b(t) \Pb\{ h(X) > t \} \to  c_\mu \, \E \left\{h(\Theta)^\beta\right\}, \ \ \mbox{ as } t\to +\infty,
$$
where $c_\mu$ and $\Theta$ are given by \cref{thm:polar}.
\end{lemma}

We end this section with the construction of a multivariate regularly varying vector $X$ that can realize all possible asymptotic dependence structures. \commb{The following lemma also provides a constructive way to prove the converse claim (ii) in \cref{thm:tail-index}.}

\begin{lemma}[Generalized Breiman's lemma] \label{lem:Breiman} 
Let $Y$ be a random variable independent of a random vector $W=(W_i)_{i=1}^d$. 
Suppose $Y$ is non-negative and it has a heavy, regularly varying right tail, namely $\Pb[ Y > t ] \sim L(t) t^{-\beta}$ for some slowly varying function $L$. 
Further, suppose $\E[ \|W\|^{\beta + \epsilon} ] <\infty$ for some $\epsilon>0$. 
Then, it holds that $ X:= (Y W_i)_{i=1}^d$
is multivariate regularly varying with exponent $\beta$.
Its angular measure in \eqref{e:thm:polar-ii} is identified by 
\begin{equation}\label{e:lem:Breiman}
\Pb[ \Theta \in A ] = \frac{1}{\E \|W\|^\beta } \E\left\{ 1_A\Big(\frac{W}{\|W\|}\Big) \|W\|^\beta\right\}
\end{equation}
for every Borel set $A \subset S_{\|\cdot\|}$. 
\end{lemma}

For this result, see, e.g., Corollary 2.1.14 in \citet{kulik:soulier:2020}. 
This is a multivariate extension of Breiman's lemma \citep[Lemma 1.4.3 in ][]{kulik:soulier:2020}, which was 
originally formulated for $d=1$ and $\beta \in (0,1)$ \citep[Proposition 2 in][]{breiman:1965}.
\commb{To show claim (ii) of \cref{thm:tail-index}, let $\mu$ be an arbitrary measure that satisfies \cref{e:thm:tail-index-mu-scaling}. 
Let $W \sim \sigma$ with angular measure $\sigma$ identified by \eqref{eqs:polar} and let $Y$ be Pareto with $\Pb[Y>t] = 1/t^\beta$ for $t\ge 1$. Then, by \cref{lem:Breiman} we have $X \sim \RV_\beta(\R^d,b(\cdot),\mu)$ with $b(t) = c_{\mu} t^{\beta}$.}

\subsection{Examples of multivariate regular variation} \label{sec:examples}

Multivariate regular variation is typically the rule rather than the exception for random vectors with heavy-tailed marginals. 
To make this intuition concrete, in this section we describe some examples that satisfy multivariate regular variation; see also \cref{sec: examples supp} of the Supplementary Material for more instances. 
To the best of our knowledge, there is no simple, non-pathological construction of a heavy-tailed random vector that is not multivariate regularly varying.

 \begin{example}[multivariate $t$-distribution] \label{ex:mult-t} 
Let $\nu>0$ and $G$ be a Gamma-distributed random variable with shape $\nu/2$ and rate $1/2$. Also, let $W\sim {\cal N}(0,\Sigma)$ be independent 
 of $G$. Then the random vector $ X :=  W / (G/\nu)^{1/2} $
follows a multivariate $t$-distribution with $\nu$ degrees of freedom and shape $\Sigma$.
 Since $Y:= (G/\nu)^{-1/2}$ is heavy-tailed with exponent $\nu$, the multivariate $t$ model is a particular instance of Breiman's construction:  
\cref{lem:Breiman}
implies that  $X=Y W \in \RV_\nu(\R^d,b(\cdot),\mu)$ with angular measure $\sigma$ given by \cref{e:lem:Breiman}. 
Unless $W$ is concentrated on a lower-dimensional 
subspace, the support of $\sigma$ is the {\em entire unit sphere}. 
In fact, the upper tail dependence coefficient of the 
$t$-copula, namely $\lambda(X_i,X_j)$, can be written as
\begin{equation}\label{e:ex:mult-t-lambda}
\lambda(X_i,X_j) = 2 F_{t_{\nu+1}}\left( - \left\{\frac{(\nu+1)(1-\rho_{ij})}{(1+\rho_{ij})}\right\}^{1/2} \right),
\end{equation}
where $\rho_{ij} = {\rm Corr}(W_i,W_j)$ and $F_{t_{\nu+1}}$ is the distribution function of the standard univariate
$t$-distribution with $(\nu +1)$ degrees of freedom; see, e.g., \citet[][p.~64]{joe2015dependence}. 
Thus, $X_i$ and $X_j$ are always asymptotically dependent, even when $\rho_{ij}=0$; 
for any fixed $\rho_{ij}$, $X_i$ and $X_j$ approach asymptotic independence only when $\nu \rightarrow +\infty$, upon which the multivariate $t$-distribution converges to a multivariate normal. 
\end{example}

\begin{example}[heavy-tailed factor models] \label{ex:linear_factor_model} 
Let $\beta>0$ and $Z_1, \dots, Z_p$ be iid {\em non-negative}\footnote{The example extends to random variables with two-sided heavy tails, but
the formula for the angular measure is slightly more involved.} random variables with Pareto-type tails, i.e., $\Pb[ Z_j > t] \sim t^{-\beta}, \ \mbox{ as }t\to +\infty.$
Let $A \in \R^{d \times p}$ be an arbitrary constant matrix with non-zero columns $a_1, \dots, a_p$. 
Then, with $Z:= (Z_j)_{j=1}^p$, define
$$
X := A Z \in \RV_\beta(\R^d, b(t)=t^{\beta},\mu),
$$
where the associated angular measure is given by
\begin{equation}\label{e:ex:linear_factor_model}
\sigma(B) = \frac{1}{\sum_{k=1}^p \|a_k\|^\beta} \sum_{j=1}^p \|a_j\|^\beta \,1_B\Big(\frac{a_j}{\|a_j\|}\Big),
\end{equation}
where $B$ is any Borel set in $S_{\|\cdot\|}$; 
see also Corollary 2.1.14 in \citet{kulik:soulier:2020} for a more general result.
\end{example}

\cref{ex:linear_factor_model} illustrates the {\em single large jump heuristic} for sums of independent 
heavy-tailed  factors:  
the vector $X = Z_1 a_1 + \cdots + Z_p a_p$ is {\em extreme} in norm when
one and only one of the independent factors is extreme.  
Hence, as $t \to +\infty$, the angular distribution of $X/\|X\|$ given $\|X\|>t$ converges to a discrete measure with point-masses given
by the directions $a_j/\|a_j\|$ ($j=1,\dots,p$) and each corresponding probability proportional to $\|a_j\|^\beta$.

\subsection{A general approach to calibrating heavy-tailed combination tests} \label{sec:MRV-calibration}
Let $P=(P_i)_{i=1}^d$ be a random vector with Uniform$(0,1)$ marginal distributions, which consists of $p$-values under a null hypothesis. 
Consider a heavy-tailed distribution $F$ with tail index 1, namely $
\bar{F}(x) :=   1-F(x) \sim a/x,\ \mbox{ as }x\to +\infty$
for $a>0$. Let us transform the $p$-values into $X = (X_i)_{i=1}^d$ by \cref{eq:X}.
Given a vector of weights $w_i\ge 0$ such that $\sum_{i=1}^d w_i = 1$, 
  consider the statistic
  \begin{equation}\label{e:Tw}
  T_w(X):= \sum_{i=1}^d w_i X_i.
  \end{equation}
Thus, small $p$-values correspond to large values of $T_w$. 
When $\bar{F}(x) = \tfrac{1}{2} - \arctan(x) / \pi \sim 1 / (\pi x)$ is the standard Cauchy distribution, this leads to the Cauchy combination test \citep{liu:xie:2020}. 
When $\bar{F}(x) = x^{-1}$ is the standard Pareto with unit tail index, this recovers a test equivalent to the {\em harmonic mean} $p$-value \citep{wilson2019harmonic,good1958significance}. 
In both cases, either under independence or asymptotic independence of $X_1, \dots, X_d$, it has been shown that 
\begin{equation}\label{e:Tw-calibrated}
\frac{\Pb\{ T_w(X)>t \}}{\Pb(X_1>t)} \to 1, \quad t \to +\infty.
\end{equation}

As noted in \cref{r:ai}, the bivariate copula conditions in \citet{liu:xie:2020,long2023cauchy} imply that $X_1, \dots, X_d$ are asymptotically independent and the vector $X$ is multivariate regularly varying with tail index 1 when $F$ is Cauchy or Pareto. It follows that the exponent measure $\mu$ of $X$ is the same as that of a vector composed of \emph{iid copies} of $X_1$. This underlies the calibration of $T_w(X)$, for which the dependence among $X_1, \dots, X_d$ can be ignored. 

However, \cref{e:Tw-calibrated} need not hold anymore when $X$ is regularly varying but $X_1, \dots, X_d$ are \emph{asymptotically dependent}. \commb{The} next result computes the limit in terms of the angular probability measure. We use $(\cdot)_+$ to denote the positive part of a variable. 
  
\begin{proposition} \label{p:general} 
Let $X = (X_i)_{i=1}^d\in \RV_\beta(\R^d,b(\cdot),\mu)$ such that for $i=1,\ldots,d$, it holds that 
\begin{equation}\label{e:p-general}
   b(t) \Pb( X_i > t ) \to c>0, \quad t \to +\infty.
\end{equation}
Let $\Theta \in S_{\|\cdot\|}$  be distributed according to the angular
probability measure $\sigma$ of $X$.
Then, we have $\E\{ (\Theta_1)_+^\beta \}=\cdots= \E\{(\Theta_d)_+^\beta\} >0$ and for any $w_1, \dots, w_d \geq 0$ such that $\sum_{i=1}^{d} w_i > 0$, 
  \begin{equation}\label{e:p-general-limit}
  \frac{\Pb\{T_w(X)>t\}}{\Pb(X_1>t)} \to \frac{1}{\E\{(\Theta_1)_+^\beta\}}  \E \Big\{\Big( \sum_{j=1}^d w_j \Theta_j\Big)_+^\beta\Big\}, \quad t \rightarrow +\infty. 
  \end{equation}
  Moreover, the above calibration constant is at most $1$ when $\beta\ge 1$.
  \end{proposition}
  
We remark that \cref{p:general} is not new: the limit behavior of a sum of {\em dependent} 
  heavy-tailed variables has been considered in 
  the context of financial or insurance risk. 
  For example, \citet{barbe:fougeres:genest:2006} establishes formulae similar to
  \eqref{e:p-general-limit}. See also Theorem 4.1 in \citet{embrechts:lambrigger:wuthrich:2009} 
  and \citet{yuen:stoev:cooley:2020} in the context of quantifying extreme Value-at-Risk. \commb{Nonetheless, we provide in the Supplementary Material an alternative proof of this proposition using \cref{l:homogeneous}.}
  
\subsection{Universal calibration and honesty}
For the rest of this paper, we identify any heavy-tailed combination test with a heavy-tailed distribution $F$ and a combination function $h$, the latter of which is typically the linear combination \cref{e:Tw} but can also take other forms. In \cref{sec:Pareto-test}, we will focus on the class of tests where $h$ is homogeneous. The following definition categorizes heavy-tailed combination tests according to their asymptotic calibration property under multivariate regular variation; compare it with \cref{def:plain}.

\begin{definition} \label{d:universal} 
Let $(P_i)_{i=1}^{d}$ be a random vector with Uniform$(0,1)$ margins. 
Let $F$ be a heavy-tailed distribution function and $h: \R^{d} \to \R_+$ be a combination function. 
Define $X_i := F^{-1}(1-P_i)$ for $i=1,\dots,d$. 
Then, the $(F,h)$-combination test is 
\[\begin{cases} \text{universally (asymptotically) calibrated}, & \quad \text{iff } \lim_{t \to +\infty} \Pb(h(X) > t) / \Pb(X_1 > t) = 1, \\
\text{universally (asymptotically) honest}, & \quad \text{iff } \limsup_{t \to +\infty} \Pb(h(X) > t) / \Pb(X_1 > t) \leq 1, \\
\text{universally (asymptotically) conservative}, & \quad \text{iff } \limsup_{t \to +\infty} \Pb(h(X) > t) / \Pb(X_1 > t) < 1, \end{cases} \]
whenever $X = (X_i)_{i=1}^{d}$ is multivariate regularly varying.
\end{definition}
  
Throughout, we omit `asymptotically' when referring to these properties. 
For the next two results, we apply \cref{p:general} to characterize the calibration of Pareto and Cauchy linear combination tests, for which we assume $X$ is multivariate regularly varying but allow $X_1,\dots,X_d$ to be asymptotically dependent. 
We first show that the Pareto linear combination test is universally calibrated regardless of the asymptotic dependence structure of $X_1, \dots, X_d$.
 
\begin{corollary}[Pareto linear combination test] \label{c:PCT}
Let $F$ be the Pareto distribution with tail index 1, namely $\bar{F}(x) = 1/x$ for $x \geq 1$. 
For any $w_1,\dots,w_d \geq 0$ with $\sum_{i=1}^d w_i=1$, the $(F,T_w)$-combination test is universally calibrated. 
\end{corollary}

\begin{proof}  
Since $X$ has positive coordinates, \cref{e:thm:polar-ii} implies $\Theta_i \geq 0$ for $i=1,\dots,d$. 
Applying \cref{p:general} with $\beta=1$, we obtain
\[ \lim_{t\to +\infty} \frac{\Pb\{T_w(X)>t\}}{\Pb(X_1>t)} = \frac{1}{\E \Theta_1 } \sum_{j=1}^d w_j \E \Theta_j =
 \sum_{j=1}^d w_j = 1,\]
where we used $\E \Theta_1 = \cdots = \E \Theta_d>0$.
 \end{proof}

In contrast, the Cauchy combination test is always honest and typically conservative.
 
 \begin{corollary}[Cauchy linear combination test] \label{c:CCT} 
Let $F$ be the Cauchy distribution, namely $\bar{F}(x) = \tfrac{1}{2} - \arctan(x)/\pi$ for $x \in \mathbb{R}$. 
For any $w_1,\dots,w_d \geq 0$ with $\sum_{i=1}^d w_i=1$, the $(F, T_w)$-combination test is universally honest, i.e.,
\[ \lim_{t \to +\infty} \frac{\Pb\{ T_w(X)>t\}}{\Pb(X_1>t)} \le 1,\]
where the equality holds if and only if $\Theta \in (-\infty,0]^d \cup [0,\infty)^d$ holds with probability one with respect to the angular measure of $X$. 
\end{corollary}

\begin{proof}
Applying \cref{p:general} with $\beta=1$, we have
     \begin{align}
         \lim_{t\to +\infty} \frac{\Pb\{T_w(X)>t\}}{\Pb(X_1>t)} = \frac{1}{\E (\Theta_1)_+}  \E \Big ( \sum_{i=1}^d w_i \Theta_i\Big)_+ \label{e: +ve part calibration limit}
     \end{align}
     By the convexity of $x \mapsto x_+$ and Jensen's inequality, we further have
     \begin{align*}
     \left (\sum_{j=1}^d w_j \Theta_j \right)_+ \leq \sum_{j=1}^d w_j (\Theta_j)_+\implies \E \left ( \sum_{j=1}^d w_j \Theta_j \right)_+ \leq \sum_{j=1}^d w_j \E (\Theta_j)_+=\rbrac{\sum_{j=1}^dw_j}\E(\Theta_1)_+=\E (\Theta_1)_+,
     \end{align*}
where we used $\E(\Theta_1)_+=\dots=\E(\Theta_d)_+ > 0$.
Thus, the limit in \cref{e: +ve part calibration limit} is upper bounded by 1. 
For the proof of the condition for equality, see \cref{sec:proof-cauchy} of the Supplementary Material.
\end{proof}

\cref{c:CCT} implies that under many dependence models, such as the multivariate $t$-copula, the Cauchy combination test is {\em strictly conservative} (see also \cref{sec:examples}). 
This corroborates the empirical findings presented in Tables~2 and~S1 of \citet{gui2025aggregating}: for $p$-values generated from a multivariate $t$-copula with an exchangeable covariance, the Cauchy combination test is conservative under smaller positive or negative correlation $\rho$; 
meanwhile, the test becomes asymptotically calibrated when $\rho \to 1$, which drives $\Theta_1,\dots,\Theta_d$ to be simultaneously positive or negative.

\begin{remark}
    \commb{The universal calibration of a linear combination test is not determined by whether the support of $X_i$ is contained in $\R_+$. For example, if $X_i$ has a light lower tail, say Gaussian, but a $\RV_1$ upper tail, any admissible exponent measure will be supported on the positive orthant, leading to universal calibration.}
\end{remark}
\medskip The function $T_w(\cdot)$ is a special case of  \emph{homogeneous} combination functions, which can be studied with the same tool. 
The next result extends \cref{p:general} with virtually the same proof.
  
\begin{corollary}\label{c:extend}
Let $h:\R^d\to \R_+$ be a continuous and $1$-positively-homogeneous function. Then, under the assumptions of \cref{p:general}, we have
\[\frac{\Pb\{h(X)>t\}}{\Pb(X_1>t)} \to \frac{1}{\E\{ (\Theta_1)_+^\beta \}} 
  \E\{ h(\Theta)^\beta\}, \quad t \to +\infty.\]
\end{corollary}

Many commonly used methods for combining $p$-values or test statistics, such as $\min$, $\max$ and the generalized means $(d^{-1} \sum_{i} x_i^r)^{1/r}$, are such homogeneous functions. \commb{The classical Tippett's combination method \citep{tippett1931methods,dunn1958estimation,vsidak1967rectangular} is one such example. Below, we present its calibration properties in \cref{c:extend}.}

\begin{example}[\commb{Tippett's method}]
\commb{We use symbols `$\wedge$' and `$\vee$' to denote the minimum and the maximum, respectively.
Consider rejecting the global null when $P_{\min} := \wedge_{i=1}^{d} P_i$ falls below the critical value $t_{\alpha} = 1 - \exp\{d^{-1} \log(1-\alpha)\}$, which is set according to
$ 1 - (1-t_{\alpha})^{d} = \alpha. $
By construction, this method is exact if $P_1,\dots,P_d$ are independent and uniformly distributed under the null \citep{tippett1931methods,dunn1958estimation,vsidak1967rectangular}. 
In fact, this test is also a heavy-tailed combination test. 
To see this, consider the standard Fr\'echet distribution with shape 1, namely $ F(x) = \exp(-1/x), \ x>0,$
which has a Pareto tail $\bar{F}(x) \sim 1/x$ as $x \rightarrow +\infty$. Transforming the $p$-values to Fr\'echet scale, i.e., $X_i=F^{-1}(1-P_i)$, the heavy-tailed statistics are then combined through the maximum divided by $d$:
\[ h_T(X):=\frac{1}{d} \bigvee_{i=1}^d X_i = - \frac{1}{d \log (1-P_{\min})}, \]
which is a continuous, 1-positively-homogeneous function of $X$. Thus, the combined statistic leads to a rejection if 
\[ h_T(X) > F^{-1}(1-\alpha) = - 1 / \log (1-\alpha) \iff P_{\min} < t_{\alpha}. \]}

\commb{If we assume that $X$ is multivariate regularly varying, it is easy to see that one can analyze the calibration properties of Tippett's method, i.e., $(F,h_T)$-combination test, using \cref{c:extend}. 
This test is a special case of the Fr\'echet max-linear combination test, which we study in the next result; see \cref{sec: fct calib} in the Supplementary Material for its proof.} 
\commb{\begin{theorem}[Fr\'echet max-linear combination test] \label{thm:frechet}
Let $X=(X_i)_{i=1}^{d}$ be a random vector that is marginally distributed as the standard Fr\'echet distribution with shape 1, namely $F(x) = \exp(-1/x)$ for $x > 0$. 
Given any $w_1,\dots,w_d > 0$, consider $h_{\vee,w}: \R^{d} \rightarrow \R$ defined as 
\[ h_{\vee,w}(x):=\frac{\bigvee_{i=1}^{d} w_i x_i}{\sum_{i=1}^{d} w_i}. \]
We have the following results. 
\begin{enumerate}
\item If $X_1, \dots, X_d$ are independent, we have $h_{\vee,w}(X)\stackrel{d}{=} X_1$. 
\item If $X$ is multivariate regularly varying, the $(F, h_{\vee,w})$-combination test is universally honest, i.e.,
\[ \lim_{t \to +\infty} \frac{\Pb\{h_{\vee,w}(X) > t\}}{\Pb(X_1 >t)} = \lim_{t \to +\infty} t\, \Pb\{h_{\vee,w}(X) > t\} \leq 1,\]
where the equality holds if and only if $X_1, \dots, X_d$ are asymptotically independent. 
\end{enumerate}
\end{theorem}}

\commb{The Fr\'echet max-linear combination test can be applied to $p$-values produced by multiple data splitting, which we discuss in \cref{sec: supp_application_data_split} of the Supplementary Material. 
\cref{thm:frechet} also implies the following property of Tippett's method. 
\begin{corollary} \label{cor:tippet}
Suppose $P_i \sim \text{Uniform}(0,1)$ for $i=1,\ldots,d$ and $X$ with $X_i=F^{-1}(1-P_i)$ is multivariate regularly varying.
Then, Tippett's method, also known as the Dunn--\v{S}id{\'a}k correction, is universally honest.
Further, it is conservative except when the copula between every pair of $p$-values is lower-tail independent.
\end{corollary}
\begin{proof}
With $h_{T} = h_{\vee,w}$ for $w = (1/d,\dots,1/d)$, the second part of \cref{thm:frechet} shows $h_T(X)$ is universally honest. Further, it is conservative unless $X_1,\dots, X_d$ are asymptotically independent, or equivalently, every pair of $p$-values are independent in the lower tail.
\end{proof}}

\commb{This result complements the existing results on the $h_T$ test under dependence: it has been shown to be honest at every level $\alpha < 1/2$ under any multivariate normal copula \citep{vsidak1967rectangular} and under $\mathrm{MTP}_2$ dependence \citep{sarkar1998some}.}

\end{example}

\section{Characterizing universal calibration} \label{sec:Pareto-test}
\subsection{On integrals under linear constraints}
\label{sec:integral_functionals}
In the previous section, we showed that the Pareto linear combination test is universally calibrated regardless of the dependence structure of the $p$-values, provided that the transformed vector $X$ is multivariate regularly varying. 
In this section, we will characterize this property for the class of $(F,h)$-combination tests when $h$ is homogeneous and $\commb{\bar{F}=1-F\in \RV_1}$ by showing that the Pareto linear combination test is the only test in this family that achieves universal calibration.  
To prove this, we first establish an auxiliary result on integrals under linear constraints.

Let $(S,{\cal S})$ be a measurable space and let ${\cal M}(S)$ be the set of all finite positive measures on the space.
We also use ${\mathbb B}_+(S)$ to denote the class of all real-valued, non-negative, bounded measurable functions on the space. 
For $\varphi \in {\cal M}(S)$ and $f\in{\mathbb B}_+(S)$,
we shall write
$$
 (f,\varphi) :=\int_S f(x) \varphi(dx).
$$

\begin{definition}[Anti-dominance condition]
We say that a finite set of non-negative 
functions ${\cal G}:= \{g_i,\ i=1,\ldots,d\} \subset {\mathbb B}_+(S)$ satisfies the 
anti-dominance condition if 
for all ${\cal I},\ \emptyset \neq {\cal I} \subsetneq \{1,\ldots,d\}$, we have
$$
\sum_{i\in {\cal I}} \lambda_i g_i(\cdot ) \nleq 
\sum_{j\in {\cal I}^c} \lambda_j g_j(\cdot), 
$$
for all $\lambda_1,\ldots,\lambda_d\ge 0$ such that 
$\sum_{i\in {\cal I}}\lambda_i>0$.
\end{definition}

A finite set of functions $\mathcal{G}$ satisfies the condition above if no subset of the functions can be dominated by the complementary subset of functions, in terms of non-negative linear combinations. Our characterization of universal calibration relies on the following general result, which may be of independent interest; see \cref{sec:char} of the Supplementary Material for its proof. 

\begin{theorem}\label{thm:characterization}
Let ${\cal G}= \{g_1,\ldots,g_d\}$ be a finite set of functions in ${\mathbb B}_+(S)$.  For a constant $c>0$, define the set of positive finite measures:
$$
{\cal M}_c({\cal G}):= \{ \varphi \in {\cal M}(S)\, :\, (g,\varphi) = c,\ \forall g\in {\cal G}\}.
$$
Suppose that for some $\{x_1,\ldots,x_d\} \subset S$,  
the matrix $G = (G_{ij})_{d \times d} := (g_i(x_j))$ is non-singular and  
the vector $(1,\dots,1)^{\T} \in \R^d$ belongs to the interior of the cone
\begin{equation}\label{a:interior_condition}
    G (\R_+^d):= \{ y:\, y =  Gz,\  z \in \R_+^d\}.
\end{equation}
If for some $h\in {\mathbb B}_{+}(S)$, $(h,\varphi) = c$ holds for all $\varphi \in {\cal M}_c({\cal G})$, then we have
\begin{equation}\label{e:thm:characterization}
h(\cdot) = \sum_{i=1}^d \lambda_i g_i(\cdot), \quad \text{with } \lambda \in \mathbb{R}^{d} ~\text{ such that } \sum_{i=1}^d\lambda_i = 1.
\end{equation}
Additionally, if $\mathcal{G}$ also satisfies the anti-dominance condition, then 
\eqref{e:thm:characterization} holds with $\lambda \in \mathbb{R}_{+}^{d}$.
\end{theorem}

\subsection{Characterization} \label{sec:universal}
We now characterize universal calibration for the family of $(F,h)$-combination tests where $h$ is homogeneous. Since $(F,h)$ and $(F(\cdot / c), h / c)$ for any constant $c>0$ lead to equivalent combination tests, 
without loss of generality, we will assume \commb{$\bar{F}(x) \sim x^{-1}$ as $x \to +\infty$.}

\begin{theorem}\label{t:unique_univ_calibrated}
Let $F$ be a heavy-tailed distribution function such that $\bar{F}(x) \sim 1/x$ as $x \to +\infty$ and \commb{$F(0-):=\lim_{x \to 0-}F(x)=0$}. 
Let \commb{$h: \R^{d}_+ \to \R_+$} be a continuous, 1-positively-homogeneous function.
Then, the $(F,h)$-combination test is universally calibrated if and only if 
\[ h(x) = \sum_{i=1}^d w_i x_i,\; \forall x\in\R_+^d,\]
for some $w_1, \dots, w_d \ge 0$ such that $\sum_i w_i = 1$.
\end{theorem}

The proof of this theorem relies on the following lemma, which itself is proved in \cref{sec:supp pf of univ calib lemma} of the Supplementary Material. We use $\Delta^{d-1}$ to denote the unit simplex in $\R^d$. 

\begin{lemma}\label{l:eq_univ_calib}
Suppose $F$ and $h$ satisfy the conditions in \cref{t:unique_univ_calibrated}. 
The $(F,h)$-combination test is universally calibrated if and only if \commb{for any probability measure $\sigma$ on $\Delta^{d-1}$}, we have
 \begin{equation}\label{e: marginal_exp}
\E_{\Theta\sim\sigma} \Theta_i=1/d, \quad i=1,\dots,d \implies d\, \E_{\Theta\sim\sigma} h\rbrac{\Theta}=1.
\end{equation}
\end{lemma}

\begin{proof}[of \cref{t:unique_univ_calibrated}] 
The `if' part is proved by \cref{c:PCT}. 
We now prove the `only if' part.
By \cref{l:eq_univ_calib}, it suffices to show that if \eqref{e: marginal_exp} holds for every probability measure $\sigma$ on $\Delta^{d-1}$, then the continuous, 1-positively-homogeneous function $h(x)$ must be of the form $\sum_{i=1}^{d} w_i x_i$ for some weights $w \in \Delta^{d-1}$.
To this end, we apply \cref{thm:characterization} with $S := \Delta^{d-1}$ and $\mathcal{G}:=\{g_1,\ldots,g_d\}$, where each $g_i$ is the coordinate function $g_i(x) := x_i$.

In the context of \cref{thm:characterization}, the probability measures that satisfy the left-hand side of 
    \eqref{e: marginal_exp} are precisely given by 
    $$
    \mathcal{M}_{1/d}(\mathcal{G}):= \{ \varphi \in {\cal M}(\Delta^{d-1}):\, (g,\varphi) = 1/d,\ \forall g\in {\cal G}\}.
    $$
    Indeed, since $(g_i,\varphi)=1/d$ and $\sum_i g_i(x) = \sum_i x_i = 1$ for every $x \in \Delta^{d-1}$, we have $1=\sum_{i=1}^d (g_i,\varphi) = (1,\varphi) = \varphi(\Delta^{d-1})$, which implies that every $\varphi\in {\cal M}_{1/d}({\cal G})$ is a probability measure.
Let us check the conditions for applying the theorem. For $i=1,\dots,d$, take $x_i:=e_i$, the $i$th unit vector in $\R^d$. Then, we have $G = I_d$ and the cone $G (\R_+^d) = \R_+^{d}$, whose interior contains $(1,\dots,1)^{\T}$. Furthermore, $\commb{\mathcal{G} = \{g_1, \dots, g_d\}}$ satisfies the anti-dominance condition. 
Hence, for any $h$ that satisfies $(h, \varphi) = 1/d$ for every $\varphi \in \mathcal{M}_{1/d}(\mathcal{G})$, it holds that $h(x) = \sum_{i=1}^{d} w_i x_i$ for some $w \in \mathbb{R}_+^{d}$ with $\sum_i w_i = 1$. 
\end{proof}

In light of this theorem and the conservativeness of the Cauchy combination test shown in \cref{c:CCT}, a simple fix is to use only the positive side of Cauchy, i.e., let $F$ be the distribution function of the absolute value of a Cauchy variable. We call this modified combination test Cauchy+. The Cauchy+ combination test is universally calibrated and should behave similarly to the Pareto combination test. Indeed, this has also recently been suggested by \citet{liu2025heavily}.
\begin{remark}
    \commb{\cref{t:unique_univ_calibrated} can be generalized to indices $\beta>0$ as a corollary of \cref{t:unique_univ_calibrated} itself. Indeed, by transforming the $\beta$-Pareto marginals of an $\RV_{\beta}$ vector to standard $1$-Pareto, one recovers an $\RV_1$ vector with standardized marginals and \cref{t:unique_univ_calibrated} applies. In this case, we will have that the $(F,h)$-combination test is universally calibrated if and only if $h$ is a weighted $\beta$-power mean, i.e.,
    $h(x)=(\sum_i w_i x_i^\beta)^{1/\beta},\; \forall x\in \R_+^d, \text{ for some }w=(w_i)_{i=1}^d\in \Delta^{d-1}. $
      }
\end{remark}

\section{Simulation studies} \label{sec:num}
We use numerical simulations to study the calibration and power of four combination tests: Pareto, Cauchy, Cauchy+ and Fr\'echet. As discussed in \cref{sec:universal}, Cauchy+ is a simple improvement of Cauchy by taking $F$ to be the distribution of the absolute value of a Cauchy random variable.
R code for reproducing the simulations can be found at \url{https://github.com/parijatch/Universal_Calibration_of_PCTs}.
\subsection{Calibration}\label{sec:calib}

First, we numerically examine the calibration of combination tests. As shown respectively in \cref{c:PCT,c:CCT,thm:frechet}, Pareto is asymptotically calibrated, while Cauchy and Fr\'echet are asymptotically honest and typically conservative. 
Further, we expect Fr\'echet's type-I error to approach the nominal level when the $p$-values are less dependent near zero. Finally, we expect Cauchy+ to behave similarly to Pareto. 

We generate $p$-values from a multivariate $t$-copula, which is multivariate regularly varying. 
Consider a random vector $(T_1, \dots, T_{d})^{\T} \sim t_{\nu}(0, \Sigma)$ with two types of shape matrix 
\begin{equation} \label{eqs:t-sigma}
\Sigma_{\text{autoreg}} := ( \rho^{|i-j|})_{d\times d}, \quad \Sigma_{\text{exch}} := (\rho^{\mathbb{I}_{i\neq j}})_{d\times d},
\end{equation}
which are then converted to \commb{one}-sided $p$-values $\commb{P_i := 1 - F_{t_\nu}(T_i)}$ for testing the location. 
For all $\nu>0$, $T_1, \dots, T_d$ are in fact tail-dependent even when $\Sigma$ is a diagonal matrix; see \eqref{e:ex:mult-t-lambda}.
The degree of tail-dependence vanishes as $\nu\to \infty$, provided that $\Sigma$ is non-degenerate, which aligns with the asymptotic independence of any non-degenerate multivariate normal distribution.

\begin{figure}[htb]
\centering
\includegraphics[width=0.9\textwidth]{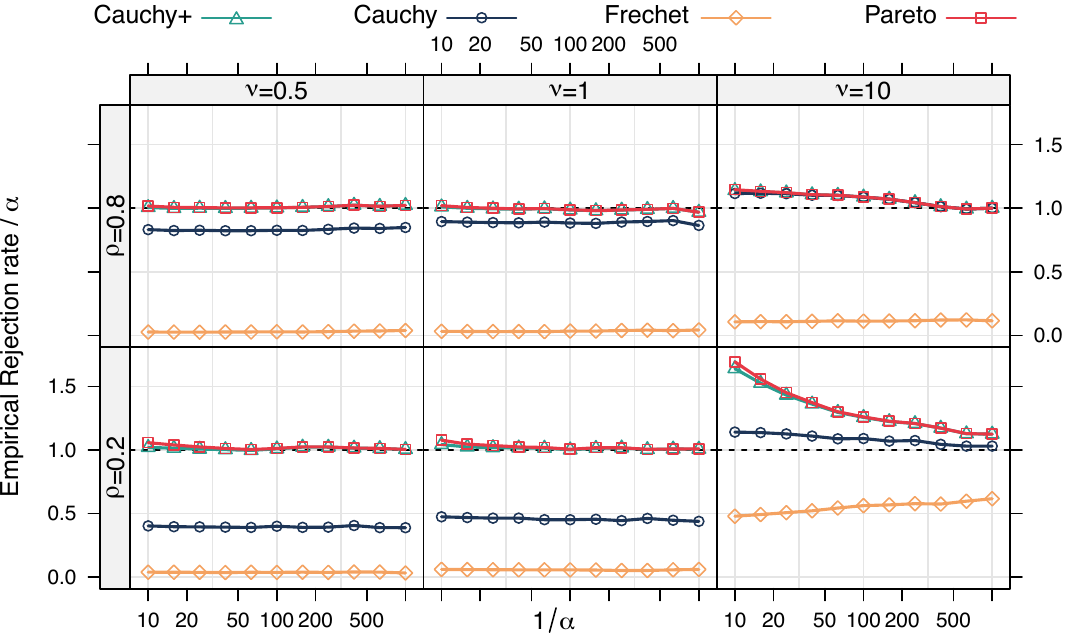}
\caption{Type-I error relative to the nominal level of combination tests under a \commb{$d=100$} dimensional multivariate $t$-copula with $\nu$ degrees of freedom and an exchangeable shape matrix $\Sigma=$
$(\rho^{\mathbb{I}_{i\neq j}})_{d\times d}$. \commb{The curves of Pareto and Cauchy+ overlap almost completely. The results are computed from $5\times10^5$ replications and the standard errors are negligible.}}
\label{fig:exch}
\end{figure}

\commb{\cref{fig:exch} reports the relative type-I error $\hat{\alpha} / \alpha$ as a function of $1/\alpha$ under $d=100$, $\rho \in \{0.2,0.8\}$ and $\nu \in \{0.5,1,10\}$ for the exchangeable $\Sigma$; a similar result under the autoregressive $\Sigma$ can be found in \cref{sec: supplement simulation} of the Supplementary Material. The weights were chosen to be uniform, i.e., $w_i=1/d$ for all $i=1,\ldots,d$. 
For the setting with $\rho=0.2$ and $\nu=1$, we also provide in the left part of \cref{fig:pairs} pairwise scatterplots of the $p$-values combined using different methods.}

The results match what our theory predicts: Pareto and Cauchy+, performing almost identically, maintained the type-I error close to $\alpha$, except when $\nu$ is large and $\alpha$ is not sufficiently small. 
\commb{Yet, for larger $\nu$ we still observe the predicted phenomenon of asymptotic calibration as $\alpha$ decreases.}
Meanwhile, Fr\'echet can be rather conservative and only approaches $\alpha$ when $\rho$ is small and $\nu$ is large, upon which the $t$-copula is close to independence. 
For most settings, Cauchy tends to be conservative but approaches calibration as $\rho$ or $\nu$ increases, which happens as the angular measure concentrates on the positive and negative orthants.

\begin{remark}
The phenomenon that the Pareto combination test has $\hat{\alpha} / \alpha > 1$ at finite $\alpha$ and larger $\nu$ is related to findings in \citep{chen2025unexpected,chen2026subuniformity}. 
From their results it follows that if $X_1, \dots, X_d$ are iid samples drawn from a Pareto distribution with tail index 1, $X_1$ is stochastically dominated by any convex combination of $X_1, \dots, X_d$. In particular, this implies that 
\begin{equation*}
\frac{\Pb\rbrac{\sum_{i} w_i X_i>1/\alpha}}{\Pb\rbrac{X_1>1/\alpha}}>1, \quad 0<\alpha<1. 
\end{equation*}
\end{remark}

\begin{figure}[htb]
\centering
\includegraphics[width=0.45\textwidth]{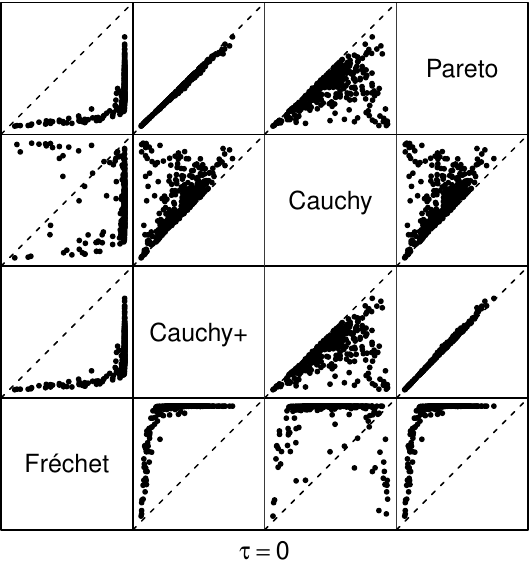}
\hspace{2em}
\includegraphics[width=0.45\textwidth]{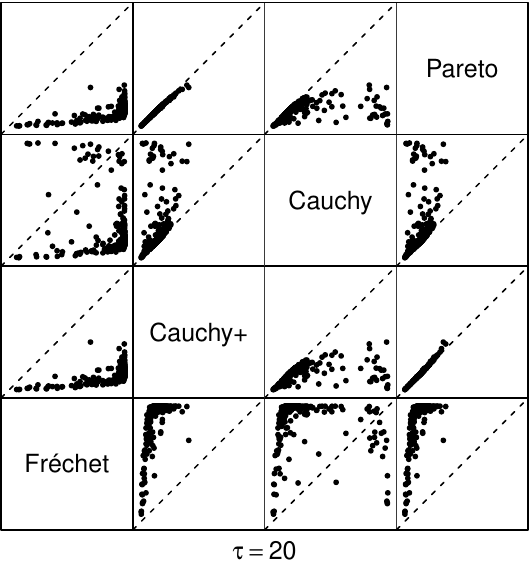}
\caption{\commb{Scatterplot matrices of the combined $p$-values from the Pareto, Cauchy, Cauchy+ and Fr\'echet combination tests.
Each off-diagonal panel compares the method named in its row (vertical axis) with the one named in its column (horizontal axis), one point per replicate; both axes span the unit interval and the dashed line is the identity.
The original $p$-values follow a 100-dimensional multivariate $t$-copula under $\nu=1$ and an exchangeable shape matrix with $\rho = 0.2$.
Left: a null setting ($\tau=0$) as in the lower-middle panel of \cref{fig:exch}; Right: an alternative setting ($\tau=20$) as in the upper-middle panel of \cref{fig:pow-exch}.}}
\label{fig:pairs}
\end{figure}

\subsection{Power}
We use simulations to study and compare the power of combination tests.
In the same setting as \cref{sec:calib}, we consider testing $\mathcal{H}_0: \mu = 0$ against $\commb{\mathcal{H}_a: \mu > 0}$ coordinatewise from a random vector $(T_1, \dots, T_{d})^{\T} \sim t_{\nu}(\mu, \Sigma)$. 
We choose $\Sigma = \Sigma_{\text{exch}}$ in \cref{eqs:t-sigma} with $\rho$ = 0.2; see also \cref{sec: supplement simulation} of the Supplementary Material for results under an autoregressive $\Sigma$.  
We consider alternatives $\mu = \tau \eta$, where $\commb{\eta=(1/\sqrt{d})\boldsymbol{1}_d}$ is the normalized vector of $1$s in $\R^d$ and $\tau > 0$ is a scalar to control the effect size. 
This requires a \commb{one}-sided test because $\mu$ has \commb{only positive} coordinates. Therefore, the $p$-values are computed as $\commb{P_i := 1 - F_{t_{\nu}}(T_i)}$ for $i=1,\dots,d$.  

\commb{To compare Cauchy, Cauchy+ and Fr\'echet combination tests against the Pareto combination test, we plot their empirical power relative to that of Pareto --- 
see \cref{fig:pow-exch} for the results under $\commb{\nu \in \{0.5,1,10\}}$, $\commb{d\in \{3,25,100\}}$ and $\alpha=0.05$. 
Additionally, pairwise scatterplots of $p$-values combined by different methods can be found in the right part of \cref{fig:pairs}.
These plots convey similar messages. 
In all settings, Cauchy+, which we know to be asymptotically equivalent to Pareto, indeed performs nearly identically to Pareto.
Meanwhile, Cauchy and Fr\'echet are less powerful than Pareto, especially when tail dependence is strong ($\nu=0.5,1$). 
Although the power of each combination test must approach 1 as $\tau \to +\infty$, Cauchy and Fr\'echet can still fail to detect the alternative under a relatively large $\tau$. 
To benchmark the performance of combination tests, the reader is also referred to \cref{sec: supplement simulation} of the Supplementary Material, where we plot their power relative to an oracle likelihood ratio test that serves as an infeasible upper bound.}

\begin{figure}[htb]
\centering
\includegraphics[width=0.9\textwidth]{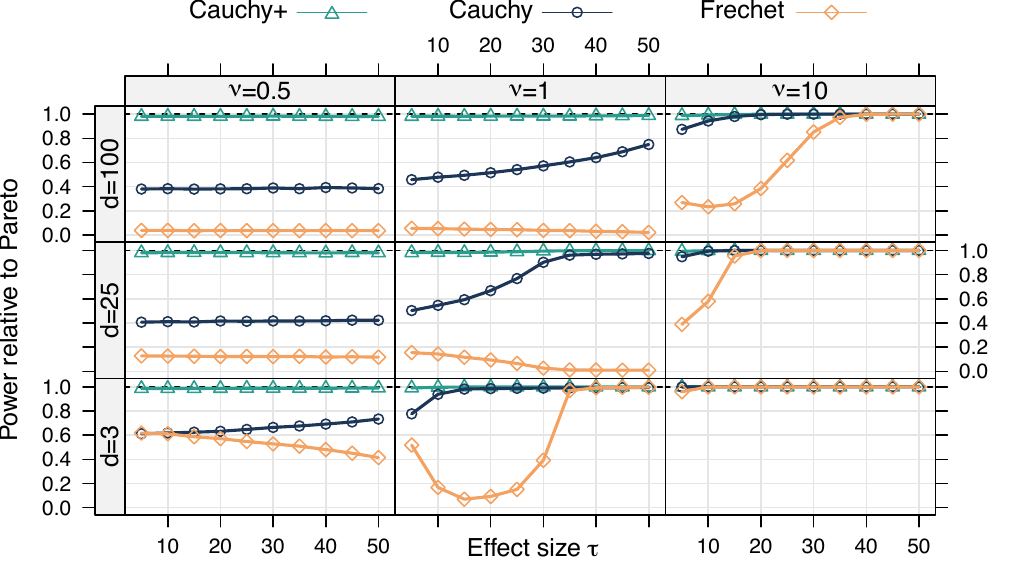}
\caption{\commb{Power of Cauchy, Cauchy+ and Fr\'echet combination tests relative to the Pareto under $\alpha=0.05$ for testing $\mu=0$ versus $\commb{\mu>0}$. Each combination test is computed from $d$ one-sided $p$-values corresponding to the coordinates of $t_{\nu}(\tau \eta, \Sigma)$, where $\Sigma$ is of exchangeable type with $\rho=0.2$. The results are computed from $5\times10^5$ replications and the standard errors are negligible.}}
\vspace{-1em}
\label{fig:pow-exch}
\end{figure}
\subsection{\commb{Pareto combination test in high dimensions}}\label{subsec: pct in high dim}

\commb{Another aspect of combination tests that is of interest is their high dimensional ($d \to \infty$) signal detection capability. In this subsection, we adopt the framework from \citet{li2020optimality} to empirically study this property of the Pareto combination test when one varies the strength of the signal and its degree of sparsity. 
Suppose we have $d$ independent observations $Y_i \sim {\cal N}(\mu_i,1)$ for $i=1,\dots,d$.
Consider testing the following sequence of nulls against the sequence of alternatives indexed by the dimension $d$:
\begin{align*}
    & \mathcal{H}_0^{(d)}: \mu_i=0, \quad i=1,\ldots,d, \text{ versus}\\
    & \mathcal{H}_a^{(d)}: \mu_i \overset{\text{iid}}{\sim} (1-\pi_d) \delta_0(\cdot) + \pi_d G_d(\cdot), \quad i=1,\ldots,d,
\end{align*}
where the alternatives take the form of a mixture model consisting of a point mass at zero and a distribution $G_d(\cdot)$ for the non-zero means. } 
\commb{We pose a certain structure on $\pi_d$ and $G_d$. 
We introduce a sparsity parameter $\gamma \in (0,1)$ and let $\pi_d=d^{-\gamma}$, so a larger $\gamma$ yields sparser signals.
Let $G_d$ be a scale family, i.e., $G_d(\cdot)=G(\cdot/s_d)$ for some fixed distribution $G$ and a scaling sequence $s_d = s_d(r,\gamma) > 0$,  where $r > 0$ is a signal strength parameter.} 

\commb{Under such a model, the goal is to study the high dimensional ($d \to \infty$) power properties of combination tests while varying $(r,\gamma)$. Previous works of \citet{li2020optimality} and \citet{donoho2004higher} determined the \emph{detection boundary} of tests such as the max test and the higher-criticism test and compared it with that of the likelihood ratio test. 
For a fixed $\alpha \in (0,1)$, consider a sequence of level $\alpha$ tests $\psi_{d}: \mathbb{R}^{d} \to \{0,1\}$ satisfying $\E_{\mathcal{H}_0^{(d)}} \psi_d(Y) = \alpha$. 
Then, the detection boundary of $\psi_d$ is defined as a function $\varrho$ such that
 \begin{equation}
 \lim_{d \to \infty}\E_{\mathcal{H}_a^{(d)}(r,\gamma)} \psi_d(Y) =
     \begin{cases}
         1, & r > \varrho(\gamma) \\
         0, & r < \varrho(\gamma).
     \end{cases}
 \end{equation}
Taking $G=\delta_{1}$ and $s_d = (2r\log(d))^{1/2}$, \citet{donoho2004higher} show that their higher-criticism test is optimal in the sense that it has the same detection boundary as the likelihood ratio test. 
Generalizing this to allow for unequal values of non-zero means, \citet{li2020optimality} show that the max test is optimal when the tails of $G$ are at least as heavy as Gaussian tails. 
In this subsection, we empirically compare the high dimensional power of the Pareto combination test to that of the max test; deriving its exact detection boundary is beyond the scope of this paper.}

\commb{Following \citet{li2020optimality}, we consider two types of $G_d$: (i) $G$ is the standard Laplace distribution and $s_{d} = r$; (ii) $G$ is the standard Cauchy distribution and $s_{d}=r(2\log(d))^{1/2}/d^{1-\gamma}$.
For our simulations, we choose $d=50000$ and $\alpha=0.05$. 
The max test and the Pareto combination test are calibrated empirically: 
we draw $10^4$ Monte Carlo samples from $\mathcal{H}_0$ and estimate $m_d(\alpha)$, the $(1-\alpha)$-quantile of $\max_{i=1}^d \abs{Y_i}$, and $t_d(\alpha)$, the $(1-\alpha)$-quantile of $\sqbrac{\sum_{i=1}^d 1/\cbrac{2(1-\Phi(\abs{Y_i}))}}/d$, to serve as their corresponding critical values.}

\begin{figure}[htb]
\centering
\includegraphics[width=1.\linewidth]{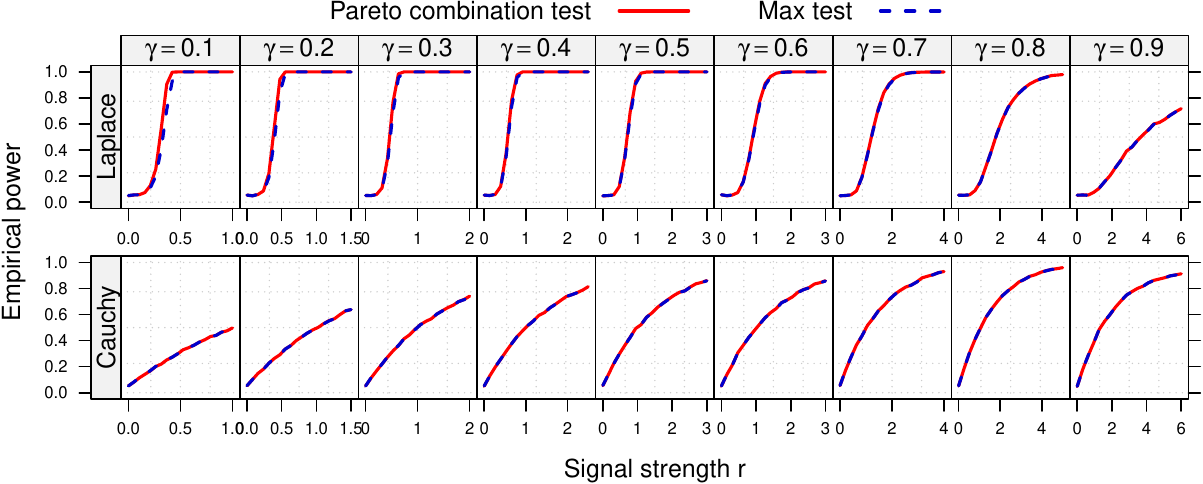}
\caption{\commb{Empirical power of the Pareto combination test and the max test in high dimensions, plotted as a function of the signal strength $r$, for $d=50000$. Rows correspond to the two signal distributions $G$; columns correspond to sparsity levels $\gamma$.}}
\label{fig:pow-high-dim}
\end{figure}
 
\commb{\cref{fig:pow-high-dim} shows the empirical power curves of the Pareto combination test and the max test, where each point on the curves is computed using $10^{4}$ Monte Carlo simulations.
One can see that the two tests perform almost identically in high dimensions across the settings we consider. 
In fact, for Laplace signals with dense alternatives ($\gamma=0.1,0.2$), the Pareto combination test marginally outperforms the max test, while the two are indistinguishable throughout the Cauchy setting; this aligns with the fact that the heavy-tailed combination is a smoothed version of the max function and thus aggregates evidence from all the $p$-values. }

\section{An application to independence testing of multidimensional physiological traits} \label{sec: application}

Projection correlation is a method for assessing the independence between two random vectors $Z \in \mathbb{R}^p$ and $Y \in \mathbb{R}^q$, based on paired realizations $\{(z_i, y_i)\}_{i=1}^n$. 
In its original form, \citet{zhu2017} proposed to use random coefficients $a \in \mathbb{R}^p$ and $b \in \mathbb{R}^q$ to obtain one-dimensional projections $(a^{\T} z_i, b^{\T} y_i)$ and then assess the association between $a^{\T} Z$ and $b^{\T} Y$ using $\{(a^{\T} z_i, b^{\T} y_i)\}_{i=1}^n$. 
This process can be repeated $d$ times: for $k=1,\dots,d$, let $r_k$ be the association statistic corresponding to coefficients $(a_k, b_k)$, which are drawn independently of the data. One may use $r_{\max} := \max_k r_k$ as the final test statistic, which can be calibrated using permutations.  

Here we consider a modified procedure: for $k=1,\dots,d$, we use $r_k$ to compute the $p$-value $P_k$ and combine $P_1,\dots,P_d$ using the Pareto linear combination test. 
Specifically, we choose $r_k$ as Kendall's rank correlation coefficient, from which the $p$-value can be derived for both independent samples and samples from complex survey designs \citep{Hunsberger2022}.

We apply this method to the 2015--2016 wave of the National Health and Nutrition Examination Survey data, which captures a wide range of health-related phenotypes of American adults. 
To assess whether vectors of related phenotypes are statistically dependent, we compute $d=100$ random projection $p$-values, where each $(a_k, b_k)$ consists of independent standard normal coordinates. 
Survey weights are used so that the results reflect the target population, and the $p$-values account for the clustered design of the survey sample. 
The final $p$-value is derived from the Pareto combination test with uniform weights. 
To control for potentially strong age and sex differences, we consider only individuals between 30 and 50 years of age, and the tests are conducted separately for females and males.  
We consider 4 multivariate phenotypes composed of the survey measures: 
4 measures of body size (height, weight, arm circumference, waist circumference) denoted as {\em bmx}, 4 measures of body composition (trunk fat mass, lean mass excluding bone, total fat mass, total bone mass) denoted as {\em dexa}, 4 measures of oral health (number of teeth that are intact, missing, replaced, and with caries) denoted as {\em den}, and 28 components of the ``standard biochemistry profile'' (based on a blood draw) denoted as {\em lab}.  All variables are standardized to have mean zero and unit variance.

\begin{remark}
\commb{The number of random projections $d$ is chosen to be $100$ because a larger number of projections improves detection capability. 
However, a larger $d$ may also slow down the convergence to calibration as $\alpha \to 0$. 
Since this is an issue that requires studying second-order regular variation, we leave it to future work.}
\end{remark}

Focusing on the extent to which blood biochemistry informs other phenotypes, we assess independence between {\em lab} and each of {\em den}, {\em bmx}, and {\em dexa} separately. 
To gauge the power and sensitivity of the testing procedure, we tested independence at a sequence of sample sizes. 
Letting $n$ be the total observed sample size, we consider samples of size $n_\ell=\lfloor n\cdot f^\ell \rfloor$ for $f=0.8$ and $\ell=0, 1, \ldots$ until $n_\ell<100$. 
As part of our sensitivity analysis, for each $n_\ell$, we sample $n_\ell$ observations uniformly without replacement 1,000 times from the total sample and report the median, 10th, and 90th percentiles of the resulting 1,000 $p$-values.  These 1,000 combined $p$-values vary both due to randomness in the subsampling and due to randomness in the projections $a_k, b_k$.  Thus, the combined $p$-values vary over replications even when $n_\ell=n$.  

\begin{table}[htb]
\caption{Summary statistics for $p$-values testing the null hypothesis of independence between blocks of variables, based on subsamples of the National Health and Nutrition Examination Survey data. \commb{$q_{10},\; q_{50}$ and $q_{90}$ represent the 10th, 50th and 90th percentiles, respectively, of the 1000 combined $p$-values resulting from subsampling and applying the Pareto combination procedure.}}
\label{nhanes}
\vspace{-0.5em}
\begin{center}
\begin{tabular}{lrrrrrrrrrrr}
& \multicolumn{5}{c}{Female} && \multicolumn{5}{c}{Male}\\\cline{2-6}\cline{8-12}
& $n\;$ & $q_{50}$ & $q_{10}$ & $q_{90}$ & Bonf && $n\;$ & $q_{50}$ & $q_{10}$ & $q_{90}$ & Bonf\\
\hline
den/lab &  620 &  0.08 &  0.04 &  0.13 &  0.35 && 648 &  0.01 &  0.01 &  0.03 &  0.04 \\
den/lab &  496 &  0.13 &  0.06 &  0.21 &  0.69 && 519 &  0.05 &  0.02 &  0.11 &  0.19 \\
den/lab &  397 &  0.14 &  0.07 &  0.23 &  0.78 && 415 &  0.07 &  0.03 &  0.14 &  0.28 \\\hline
bmx/lab &  620 &  0.00 &  0.00 &  0.00 &  0.00 && 648 &  0.00 &  0.00 &  0.00 &  0.00 \\
bmx/lab &  496 &  0.00 &  0.00 &  0.01 &  0.01 && 519 &  0.00 &  0.00 &  0.00 &  0.00 \\
bmx/lab &  397 &  0.01 &  0.00 &  0.02 &  0.02 && 415 &  0.00 &  0.00 &  0.00 &  0.00 \\\hline
dexa/lab &  620 &  0.00 &  0.00 &  0.00 &  0.00 && 648 &  0.00 &  0.00 &  0.00 &  0.00 \\
dexa/lab &  496 &  0.01 &  0.00 &  0.02 &  0.01 && 519 &  0.00 &  0.00 &  0.00 &  0.00 \\
dexa/lab &  397 &  0.01 &  0.00 &  0.02 &  0.02 && 415 &  0.00 &  0.00 &  0.00 &  0.00 \\\hline
\end{tabular}
\end{center}
\vspace{-0.5em}
\end{table}
The results for the top 3 sample sizes are summarized in \cref{nhanes}, 
with the rest provided in \cref{tab: supp_nhanes} of the Supplementary Material.
For the largest sample sizes, the null hypothesis of independence is rejected, i.e., combined $p$-value $\le 0.05$, in 5 of the 6 settings of sex $\times$ phenotype. 
The sole exception is females with oral health variables ({\em den}), where the median $p$-value is 0.08 and exceeds 0.13 in 10\% of replications.

\cref{nhanes} also reports {\em Bonf}, a Bonferroni-adjusted combined $p$-value $(d \cdot \wedge_k P_k) \wedge 1$, summarized by its median over 1,000 replications. Owing to its conservatism under positive dependence, Bonferroni consistently provides weaker evidence of multivariate dependence than the Pareto combination test, with substantially faster loss of detection power as sample size decreases. 
This is evident in the 3rd row of each sex $\times$ phenotype setting: whenever the Pareto combined $p$-value is nonzero, the corresponding Bonferroni $p$-value is at least twice as large; see also \cref{tab: supp_nhanes} of the Supplementary Material for smaller-sample results, where this effect is particularly pronounced.

Overall, this analysis provides strong evidence that the blood biochemistry panel ({\em lab}) captures multivariate information about diverse physiological traits, including body size ({\em bmx}), body composition ({\em dexa}), and oral health ({\em den}). 
The Pareto combination test is well suited to this setting, as the biochemistry variables are quantitative and often strongly right-skewed. Because different projection coefficients $(a_k, b_k)$ emphasize distinct latent factors within {\em lab}, the resulting $p$-values may exhibit tail dependence, motivating a combination method that accommodates such dependence without incurring the computational cost of permutations.
\section*{Acknowledgments}
We thank Ruodu Wang for an inspiring discussion. 
We also thank Jingshu Wang for encouraging feedback, which motivated us to formulate \cref{c:CCT}.
FRG was supported in part by NSF Grant DMS-2515385. SS and PC were partially supported by the 
NSF grant CNS/CSE-2319592 ``Collaborative Research: IMR: MM-1A: Scalable Statistical Methodology for Performance Monitoring, Anomaly 
Identification, and Mapping Network Accessibility from Active Measurements''.

\newpage
\appendix
\appendixnumbering
\crefalias{section}{appendix}
\crefalias{subsection}{appendix}
\crefalias{subsubsection}{appendix}

The Supplementary Material is organized as follows.
\cref{sec:supp-MRV} gives a brief introduction to multivariate regular variation, with extra examples presented in \cref{sec: examples supp}.
The proofs of \cref{c:CCT,l:eq_univ_calib,thm:characterization,thm:frechet} are presented in \cref{sec:supp-proofs}.
Additional results on simulation and data analysis are presented in \cref{sec: supplement simulation} and \cref{sec:supp-nhanes} respectively. 
\cref{sec: supp_application_data_split} describes an application to multiple data splitting.

\section{A brief introduction to multivariate regular variation} \label{sec:supp-MRV}

  This section reviews the fundamental concepts of multivariate regular variation needed for the paper. For comprehensive treatments, see \citet{resnick:1987,resnick:2007,kulik:soulier:2020,mikosch:wintenberger:2024,resnick:2024} and the references therein.
  
  \subsection{The space $\M_0$}

   In this section, we follow closely the seminal paper of
   \citet{hult:lindskog:2006}. Although our
   focus is on finite-dimensional Euclidean spaces, we adopt the modern language and
   the $\M_0$-convergence perspective. Thus, {\em mutatis mutandis}, all results
   in this section extend to random elements in complete separable metric spaces equipped with a
   continuous scaling action \citep[][]{hult:lindskog:2006}. Extensive expositions can be found
   in the books \citet{resnick:2007,kulik:soulier:2020}.

   Consider the Euclidean space $\R^d$. {\em Excise} its origin
   $\R_0^d:=\R^d\setminus\{0\}$ and equip it with the induced topology. Let
   ${\cal B}_0:={\cal B}(\R_0^d)$ be the Borel $\sigma$-field generated by all open sets in $\R_0^d$.

   Let $B_r(x):= \{ y\in \R^d\, :\, \|x-y\|<r\}$ denote the open ball in $\R^d$ with center $x$ and radius $r>0$. For a set $A\subset \R^d$, we write $\overline{A}$ and $A^\circ$ for the closure and interior, and let $\partial A:= \overline{A}\setminus A^\circ$ be the boundary of $A$, respectively.
   We shall say that a set $A\subset \R_0^d$ is {\em bounded away from the origin}
   (BAFO), if for some $\epsilon>0$, we have $B_\epsilon(0)\cap A = \emptyset$. That is, the BAFO
   sets are a positive distance away from $0$.

   \begin{definition}[The $\M_0$ space and $\M_0$-convergence]
   {\em (i)} A measure $\mu$ on $(\R_0^d,{\cal B}_0)$ is said to be
   {\em boundedly finite} if $\mu(A)<\infty$, for all BAFO Borel sets.  Let $\M_0:=\M_0(\R^d)$
   denote the collection of all such measures.

   {\em (ii)} For  $\mu,\mu_n \in \M_0,\ n\in\N$, we write $\mu_n\to^{\M_0}\mu$ and say $\mu_n$
   converges to $\mu$, in the $\M_0$-topology, if for all BAFO Borel sets $A$ with
   $\mu(\partial A) = 0$,
   $$
     \mu_n(A) \longrightarrow \mu(A),\ \ \mbox{ as }n\to\infty,
   $$
   where $\partial A:= \overline{A}\setminus A^\circ$ denotes the boundary of the set $A$.
   \end{definition}

     Conceptually, it is useful to view the $\M_0$-convergence as a type of weak convergence. Let
    ${\cal C}_0$ denote the class of all {\em bounded} and {\em continuous} functions $f:\R^d\to\R$
    which vanish in a neighborhood of $0$.  That is, for all $f \in \mathcal{C}_0$ there exists an 
    $ \epsilon>0$ such that $f(x) =0$ for all $x\in B_\epsilon(0)$, which means that
    $\{ |f|>0\}$ is a BAFO set $\forall f \in \mathcal{C}_0$.

    \begin{proposition}[Theorem 2.1 in  \citet{hult:lindskog:2006}]  \label{th:conv_criteria}We have that
    $\mu_n\to^{\M_0}\mu$ if and only if $\int_{\R^d} fd\mu_n \to \int_{\R^d} fd\mu$, as $n\to\infty$, for all $f\in {\cal C}_0$.
   \end{proposition}

   The notion of $\M_0$-convergence of sequences of measures can be used to define closed sets in
   $\M_0$ and hence a topology on $\M_0$.  It can be shown that this topology is in fact metrizable.
   Recall first, that for two {\em finite} Borel measures $\mu$ and $\nu$ on $\R^d$,
   the L\'evy-Prokhorov metric, is:
   $$
   \pi(\mu,\nu) := \inf\Big\{ \epsilon>0\, :\,
   \sup_{A \in {\cal B}_0} (\mu(A) - \nu(A_\epsilon))\vee(\nu(A) - \mu(A_\epsilon)) \le \epsilon\Big\},
   $$
   where $A_\epsilon:= \cup_{x\in A} B_\epsilon(x)$ is the $\epsilon$-neighborhood of $A$ and
   $x\vee y:=\max\{x,y\}$.

    Following \citet{hult:lindskog:2006}, for every $r>0$ and a boundedly finite measure
    $\mu\in\M_0$, define $\mu^{(r)}$ as the restriction of $\mu$ to $B_r(0)^c:=\R^d\setminus B_r(0)$.
    Namely, $\mu^{(r)}$ is the {\em finite measure}
    $$
    \mu^{(r)} (A) := \mu(A\setminus B_r(0)),\ \ A\in {\cal B}_0.
    $$
    Now, for every two boundedly finite measures $\mu,\nu\in \M_0$, define
    \begin{equation}\label{e:d-M0}
     d_{\M_0}(\mu,\nu):= \int_0^\infty e^{-r}
     \frac{\pi(\mu^{(r)},\nu^{(r)})}{ 1+ \pi(\mu^{(r)},\nu^{(r)})} dr.
    \end{equation}

   \begin{proposition}[cf.\ Theorems 2.3 and 2.4 in \citet{hult:lindskog:2006}] \label{p:M0-convergence}
   The functional $d_{\M_0}$ in \eqref{e:d-M0} is a metric on $\M_0$ and $(\M_0,d_{\M_0})$ is a complete
   separable metric space.  Moreover, $\mu_n\to^{\M_0}\mu$ if and only if $d_{\M_0}(\mu_n,\mu)\to 0$, as
   $n\to\infty$.
   \end{proposition}

   For a Portmanteau theorem with equivalent characterizations of the $\M_0$-convergence,
   see Theorem 2.4 in \citet{hult:lindskog:2006}. We conclude this brief review with a
   characterization of the important notion of {\em relative compactness}, which is also
   reproduced from \citet{hult:lindskog:2006}. Recall that a set of measures $M\subset \M_0$ is
   said to be relatively  compact if its closure is compact. Equivalently, an infinite subset $M$
   of the metric space $\M_0$ is relatively compact if and only if every infinite sequence
   $\{\mu_n\}\subset M$ has a converging infinite subsequence $\{\mu_{n_k}\}$, whose limit is in $\M_0$ though not necessarily
   in $M$.

    \begin{proposition}[Theorem 2.7 in \citet{hult:lindskog:2006}] \label{p:M0-tightness} A
    set of measures $M \subset \M_0$ is relatively compact in $(\M_0,d_{\M_0})$ if
    and only if for some $r_n\downarrow 0$, the following two conditions hold:

     \begin{enumerate}
         \item For all $n\in \N$, we have
         \begin{equation}\label{e:tightness_crit_1}
                        \sup_{\mu\in M} \mu\Big( \R^d \setminus B_{r_n}(0)\Big) <\infty
         \end{equation}
         
         \item For every $\epsilon>0$, there exist compact sets $C_n \subset \R^d\setminus B_{r_n}(0)$,
         such that 
         \begin{equation}  \label{e:tightnes_crit_2}
            \sup_{\mu \in M} \mu\Big( \R^d \setminus( C_n \cup B_{r_n}(0) ) \Big) <\epsilon. 
         \end{equation}
     \end{enumerate}
    \end{proposition}

   The necessity of this characterization of relative compactness essentially follows from \cref{p:M0-convergence} and Prokhorov's characterization of relative compactness for finite
   measures on complete separable metric spaces \citep{billingsley:1999}. The sufficiency is a consequence
   of Theorem 2.2 in \citet{hult:lindskog:2006} and yet again Prokhorov's criterion.

\subsection{Relative compactness of tail-measures}

  In this section, we establish a result of independent interest.  It shows that the tail-measures 
  of a random vector with regularly varying marginals are relatively compact in the $\M_0$-topology.
  As a consequence, this allows us to recover the well-known fact that asymptotic bivariate independence
  implies multivariate regular variation dating back to \citet{berman1961convergence} (cf.\ (8.100) in \citealp{beirlant2004statistics}).

\begin{proposition}\label{p:rel-compact} Let $X=(X_i)_{i=1}^d$ be a random vector. Assume that the
marginals of $X$ have regularly varying distributions.  Specifically, suppose that
for all $x>0$ and $i\in [d]$, we have
\begin{equation}\label{e:p:rel-compact}
b(t) \P[ \pm X_i > t x] \to c_\pm x^{-1},\ \ \mbox{ as }t\to\infty,  
\end{equation}
where $c_\pm \ge 0$ and $c_++c_- = 1$, for some monotone non-decreasing 
function $b$ such that $b(t)\to \infty$.

Define the rescaled tail-measures 
$$
\mu_t(\cdot):= b(t)\P[ X/t \in \cdot],\ \ t>1
$$
on $(\R_0^d, {\cal B}_0)$ and observe that $\mu_t\in \M_0$. Then:\\

{\em (i)} We  have that $b(t) \sim L(t) t,$ as $t\to\infty$ for some slowly varying function 
$L(\cdot)$. \\

{\em (ii)} The set of rescaled tail-measures  $\{\mu_t,\ t>1\}$ is relatively compact in the 
$\M_0$-topology.  In particular,  for every  $t_n\to\infty$, there is a measure 
$\mu \in \M_0(\R^d)$ 
and a further integer sequence $n_k\to\infty$ such that
$$
\mu_{t_{n_k}} \stackrel{\M_0}{\longrightarrow} \mu,\ \ \mbox{ as }n_k\to\infty.
$$
\end{proposition}
\begin{proof} If $t_n \not \to \infty$, then one can choose a convergent monotone subsequence. Without loss of generality assume the subsequence is increasing, i.e., $t_{n_k}\uparrow \tau <\infty$. By the monotonicity of $b$ 
one readily has $\mu_{t_{n_k}}\to^{\M_0} \mu$, as $n_k\to\infty$, for some non-zero $\mu$.  Indeed, in this case $b(t_{n_k})\to b(\tau-)$, 
and we have $\mu = \mu_{\tau-}:= b(\tau-) \P[X/\tau \in \cdot]$. (If $t_{n_k}$ is decreasing, replace $b(\tau-)$ with $b(\tau+)$) The interesting case is when $t_n\to \infty$.

For this case, we use the analogous tightness criteria for boundedly finite measures (\cref{p:M0-tightness}). By \eqref{e:p:rel-compact}, for every $x>0$, we have that
$$
\mu_t(x\cdot A_i) = b(t)\P[ X/t \in x\cdot A_i] = b(t)\P[ |X_i| > xt] \to x^{-1},\ \ \mbox{ as }t\to\infty
$$
where $A_i:=\{ u\in \R^d\, :\, |u_i|>1\}$. Take any $r_n \downarrow 0$. Then for all $n, \frac{r_n}{d}\bigcap_{i=1}^dA_i^c = \{u \in \R^d : |u_i|\leq r_n/d\; \forall\; i\} \subseteq B_{r_n}(0) \implies \mu_t\Big( \R^d \setminus B_{r_n}(0)\Big) \leq \mu_t \Big( \bigcup_{i=1}^d \frac{r_n}{d} A
_i\Big)\; \forall t.$\\
Using \eqref{e:p:rel-compact} again, there exists $M_n$ such that for all $t>M_n$,, $\mu_t\left(\frac{r_n}{d}A_i\right)< \frac{d}{r_n}+1,\; \forall\; i.$ Also, $\forall t\leq M_n,\; \mu_t\left(\frac{r_n}{d}A_i\right)=b(t)\P(|X_i|>\frac{r_nt}{d})\leq b(M_n)$ as $b$ is non-decreasing. Thus, $\forall r_n\downarrow 0 \text{ and } \forall t>1,$
\begin{align*}
    \mu_t\Big( \R^d \setminus B_{r_n}(0)\Big)& \leq  \mu_t \Big( \bigcup_{i=1}^d \frac{r_n}{d} A_i\Big) \leq \sum_{i=1}^d \mu_t\left(\frac{r_n}{d}A_i\right) \leq d\left[\left(\frac{d}{r_n}+1\right)\vee b(M_n)\right]\\
    & \implies \sup_{t>1}\mu_t\Big( \R^d \setminus B_{r_n}(0)\Big) < \infty\quad \forall r_n\downarrow 0
\end{align*}
This proves \eqref{e:tightness_crit_1} in \cref{p:M0-tightness}. For proving \eqref{e:tightnes_crit_2}, begin with fixing any $r_n\downarrow 0 \text{ and } \epsilon>0$. Define $C_{n,\epsilon}=R_n\bigcap_{i=1}^dA_i^c$ where $R_n=R_{n,\epsilon}$ satisfies the following:
\begin{enumerate}
    \item $R_n>\max\left(1,r_n, \frac{2d}{\epsilon}\right)$
    \item If $M_{\epsilon}$ is such that $\forall t>M_{\epsilon},\; \mu_t\left(xA_i\right)\leq \frac{1}{x}+\frac{\epsilon}{2d}\; \forall i\text{ \emph{and} }\forall x>1,$ then $R_n \text{ be such that } \P(|X_i|>R_n)\leq \frac{\epsilon}{db(M_{\epsilon})}\; \forall\; i .$ Note that here we use Proposition 2.4 in \citet{resnick:2007} which states that \eqref{e:p:rel-compact} holds uniformly over $x\in(b,\infty)\; \forall\; b>0.$ Here we take $b=1$ when we impose $R_n>1.$
\end{enumerate}
Observe that, $\mu_t\left(\R^d\setminus(C_{n,\epsilon}\cup B_{r_n}(0))\right)=\mu_t\left(\bigcup_{i=1}^dR_nA_i\right)\leq \sum_{i=1}^d\mu_t(R_nA_i)$.

Then, if $t>M_{\epsilon},$
\begin{align*}
    & \mu_t(R_nA_i)\leq \frac{1}{R_n}+\frac{\epsilon}{2d} < \frac{\epsilon}{d}\\
    & \text{ (using uniform convergence over }(1,\infty) \text{ and condition 1 on R)}\\
     \implies & \sum_{i=1}^d\mu_t(R_nA_i) \leq \epsilon
\end{align*}
Next, if $1<t\leq M_{\epsilon},$
\begin{align*}
    &\mu_t(R_nA_i)=b(t)\P(|X_i|>tR_n)\leq b(M_{\epsilon})\P(|X_i|>R_n) \leq \epsilon/d\\
    & \text{(using condition 2 on R)}\\
     \implies & \sum_{i=1}^d\mu_t(R_nA_i) \leq \epsilon
\end{align*}
Thus, $\forall t>1, \mu_t\left(\R^d\setminus(C_{n,\epsilon}\cup B_{r_n}(0))\right)\leq \epsilon$, which finally proves \eqref{e:tightnes_crit_2} in \cref{p:M0-tightness}, and hence the relative compactness of $\{\mu_t,\ t>1\}$ in $\M_0.$

\end{proof}

\begin{remark}
\commb{
    The assumption that $c_\pm$ is constant with respect to $i$ and that $c_+ + c_-=1$ is not necessary. All the arguments presented above work by including a factor of $c_+(i)+c_-(i)$ and taking supremum over $i$ when necessary. The assumption was made only to make the calculations cleaner.}
\end{remark}
\begin{remark}
\commb{
\cref{p:rel-compact} plays a fundamental role in our paper as it helps in developing a novel approach to prove that 
multivariate regular variation holds whenever the tail-dependence coefficients vanish.  This recovers the classical
result due to \citet{berman1961convergence} but it is more widely applicable since it shows
the relative compactness of the tail measure for {\em an arbitrary} random vector with heavy-tailed marginals. }
\end{remark}

We start with positive regularly varying random variables and later generalize to all real-valued random variables.
\begin{lemma}\label{le: equalising_thresholds}
    Let $X,Y$ be non-negative random variables in $ \RV_1(b,c)$ for some regularly varying monotone function $b(t) \to \infty \text{ as } t\to\infty \text{ and }c>0$, i.e., $\forall x>0$
     \begin{align}\label{eq: tails_X&Y}
         \lim_{t\to \infty}b(t)\P(X>tx)= cx^{-1},\ \  \text{and} \ \ 
          \lim_{t \to \infty} b(t)\P(Y>tx)= cx^{-1}.
     \end{align}
    If they are also asymptotically independent in the upper tail, i.e.,
    $$\lambda(X,Y):=\lim_{p \to 1^-}\P\rbrac{X>F_X^{-1}(p) \mid Y>F_Y^{-1}(p)}=0,$$
    then,
    \begin{align}\label{eq: asym_indep}
        \lim_{t \to \infty} \P(X>t \mid Y>t)=0.
    \end{align}
    Here $F_X,F_Y$ represent the distribution functions of $X$ and $Y$, respectively while $F_X^{-1},F_Y^{-1}$ refer to their generalized inverses.
\end{lemma}
\begin{proof}
    Let $t \in \R$ and define $p_X(t)=F_X(t),\; p_Y(t)=F_Y(t).$ Clearly, as $t\to \infty, p_X(t)\to 1^- \text{ and } p_Y(t) \to 1^-.$ Observe that $\P\rbrac{F_X^{-1}(p_X(t))<X\le t}=\P(F_Y^{-1}(p_Y(t))<Y\leq t)=0$. As a result, we have that
    \begin{equation*}
        \P\rbrac{X>t \mid Y>t}=\frac{\P\rbrac{X>t,Y>t}}{\P\rbrac{Y>t}}=\frac{\P\rbrac{X>F_X^{-1}(p_X(t)),Y>F_{Y}^{-1}(p_Y(t))}}{\P\rbrac{Y>F_Y^{-1}(p_Y(t))}}.
    \end{equation*}
     The above expressions are all well-defined for every $t$ as the denominator is never exactly zero. This is because we assumed the tail-dependence coefficient $\lambda$ to exist which implies $X \text{ and }Y$ both have supports extending to infinity, i.e.,
    $$\sup\{x\; : \;\P\rbrac{X>x}>0\}=\infty\quad \text{(same for $Y$)}.$$ 
    Next observe that due to \eqref{eq: tails_X&Y}, $X$ and $Y$ are \emph{tail equivalent}. Indeed,
    \begin{align}\label{eq: tail_equi}
        &\lim_{t\to \infty}b(t)\P(X>t)= c \text{ and } \lim_{t\to \infty}b(t)\P(Y>t)= c\nonumber \\
        &\implies \lim_{t \to \infty} \frac{\P\rbrac{X>t}}{\P\rbrac{Y>t}}=1 \text{ or } \lim_{t\to \infty} \frac{1-p_X(t)}{1-p_Y(t)}=1.
    \end{align}
    Now, if $p_X(t)\geq p_Y(t), \text{ then } F_X^{-1}(p_X(t)) \geq F_X^{-1}(p_Y(t))$
    \begin{align}\label{ineq: bd_on_conditional}
        & \implies \P\rbrac{X>F_X^{-1}(p_X(t)),Y>F_Y^{-1}(p_Y(t))}\leq \P\rbrac{X>F_X^{-1}(p_Y(t)),Y>F_Y^{-1}(p_Y(t))} \nonumber \\
        & \implies \frac{\P\rbrac{X>t,Y>t}}{\P\rbrac{Y>t}}\leq \frac{\P\rbrac{X>F_X^{-1}(p_Y(t)),Y>F_Y^{-1}(p_Y(t))}}{\P\rbrac{Y>F_Y^{-1}(p_Y(t))}}.
    \end{align}
    On the other hand, if $p_X(t)< p_Y(t), \text{ then } F_X^{-1}(p_X(t)) \leq F_X^{-1}(p_Y(t))$ so we can't use the above bound. However, we can establish a bound infinitesimally close to the last one:
    \begin{align}\label{ineq: bd_on_conditional2}
        & \frac{\P \rbrac{X>t,Y>t}}{\P\rbrac{Y>t}} = \frac{\P\rbrac{X>F_X^{-1}(p_X(t)),Y>F_Y^{-1}(p_Y(t))}}{\P\rbrac{Y>F_Y^{-1}(p_Y(t))}} \nonumber \\
        &= \frac{\P\rbrac{X>F_X^{-1}(p_Y(t)),Y>F_Y^{-1}(p_Y(t))}}{\P\rbrac{Y>F_Y^{-1}(p_Y(t))}}+\frac{\P\rbrac{F_X^{-1}(p_Y(t))\geq X>F_X^{-1}(p_X(t)),Y>F_Y^{-1}(p_Y(t))}}{\P\rbrac{Y>F_Y^{-1}(p_Y(t))}} \nonumber \\
        & \leq \frac{\P\rbrac{X>F_X^{-1}(p_Y(t)),Y>F_Y^{-1}(p_Y(t))}}{\P\rbrac{Y>F_Y^{-1}(p_Y(t))}}+ \frac{\P\rbrac{F_X^{-1}(p_Y(t))\geq X>F_X^{-1}(p_X(t))}}{\P\rbrac{Y>F_Y^{-1}(p_Y(t))}} \nonumber  \\
        & = \frac{\P\rbrac{X>F_X^{-1}(p_Y(t)),Y>F_Y^{-1}(p_Y(t))}}{\P\rbrac{Y>F_Y^{-1}(p_Y(t))}}+ \frac{p_Y(t)-p_X(t)}{1-p_Y(t)} \nonumber \\
        & = \frac{\P\rbrac{X>F_X^{-1}(p_Y(t)),Y>F_Y^{-1}(p_Y(t))}}{\P\rbrac{Y>F_Y^{-1}(p_Y(t))}}+ \frac{1-p_X(t)}{1-p_Y(t)}-1 \nonumber \\
        & \leq \frac{\P\rbrac{X>F_X^{-1}(p_Y(t)),Y>F_Y^{-1}(p_Y(t))}}{\P\rbrac{Y>F_Y^{-1}(p_Y(t))}}+ \abs{\frac{1-p_X(t)}{1-p_Y(t)}-1}
    \end{align}
Thus, combining \eqref{ineq: bd_on_conditional} and \eqref{ineq: bd_on_conditional2}, we get that for all $t,$
\begin{align}
    \P\rbrac{X>t \mid Y>t} \leq \P\rbrac{X>F_X^{-1}(p_Y(t))\; \mid \; Y>F_Y^{-1}(p_Y(t))}+ \abs{\frac{1-p_X(t)}{1-p_Y(t)}-1}
\end{align}
Now the RHS of the above converges to $0 \text{ as } t\to \infty$ because
\begin{align*}
    \lim_{t \to \infty} \P\rbrac{X>F_X^{-1}(p_Y(t))\; \mid \; Y>F_Y^{-1}(p_Y(t))}&=\lim_{p \to 1^-}\P\rbrac{X>F_X^{-1}(p) \mid Y>F_Y^{-1}(p)}\\
    &=\lambda(X,Y)=0,
\end{align*}
and because the second term goes to $0$ due to \eqref{eq: tail_equi}. Hence,
$$\lim_{t \to \infty}\P\rbrac{X>t \mid Y>t}=0$$
which proves the claim.
\end{proof}

\begin{corollary}
    Let $X,Y$ be non-negative random variables in $\RV_1(b,c_x) \text{ and } \RV_1(b,c_y)$ for some $c_x,c_y>0$ and some regularly varying monotone function $b(t) \to \infty$, respectively. Also assume that they are asymptotically independent in the upper tail. Then,
    \begin{align}\label{eq: asym_indep-cx-cy}
        \lim_{t \to \infty} \P(X/c_x>t \mid Y/c_y>t)=0
    \end{align}
\end{corollary}
\begin{proof}
    Clearly, $X\in \RV_1(b,c_x), Y\in \RV_1(b,c_y) \implies X/c_x, Y/c_y \in \RV_1(b,1)$. 
    Moreover, using the fact that $F_{X/c_x}^{-1}(p)=c_x^{-1}F_X^{-1}(p), \; F_{Y/c_y}^{-1}(p)=c_y^{-1}F_Y^{-1}(p) $, we have
    $$\lambda(X,Y)=\lambda\rbrac{\frac{X}{c_x},\frac{Y}{c_y}}=0.$$
    This, in view of \cref{le: equalising_thresholds}, completes the proof.
\end{proof}

\begin{proposition}\label{p:same_const_indep_implies_mrv}
    Let $X,Y$ be non-negative random variables in $\RV_1(b,c)$. If they are also asymptotically independent, i.e.,
    $\lambda(X,Y)=0$, then, $(X,Y) \in \RV_1(b,\mu_{iid}^+)$ where $\mu_{iid}^+$ is the limit measure concentrated on the positive axes corresponding to the random vector comprised of iid positive $\RV_1(b,c)$ random variables. 
\end{proposition}

\begin{proof}
    From \cref{le: equalising_thresholds} we know that,
    \begin{align*}
        &\lim_{t \to \infty}\P\rbrac{X>t \mid Y>t}=0\\
        & \implies \lim_{t \to \infty}\frac{\P\rbrac{X>t,Y>t}}{\P\rbrac{Y>t}}=0\\
        &\implies \lim_{t \to \infty}\frac{b(t)\P\rbrac{X>t,Y>t}}{b(t)\P\rbrac{Y>t}}=0.
    \end{align*}
    Now, due to \eqref{eq: tails_X&Y}, we have
    $$\lim_{t \to \infty} b(t)\P\rbrac{Y>t}=c>0.$$
    Combining with the previous equality, we get
    \begin{align*}
       & \lim_{t\to\infty} b(t)\P\rbrac{X>t,Y>t}=0\\
        & \implies \lim_{t\to \infty}b(t)\P\rbrac{(X,Y) \in t \cdot B_1\cap B_2}=0,
    \end{align*}
    where $B_1=[1,\infty)\times \R_{\geq0}$ and $B_2=\R_{\geq 0} \times [1,\infty)$.
    Now note that for any $\epsilon>0,\; X/\epsilon \text{ and } Y/\epsilon \in \RV_1(b,c/\epsilon)$. Thus, all the above results hold by replacing $(X,Y) \text{ by } \rbrac{X/\epsilon,Y/\epsilon}$. As a result, $\forall\; \epsilon>0,$
    \begin{align}\label{eq:zero_mass_boxes}
        \lim_{t\to \infty}b(t)\P\rbrac{(X,Y) \in t\cdot\rbrac{\epsilon (B_1\cap B_2)}}=0.
    \end{align}
    Denoting $(X,Y) \text{ by } Z,$ let $\mu_t(A):=b(t)\P\rbrac{\frac{Z}{t}\in A}$ be the rescaled tail measure of $Z$ as defined in \cref{p:rel-compact}. Thus, 
    \begin{align}\label{eq: lim_boxes_zero}
    \forall\; \epsilon>0,\;\lim_{t\to \infty}\mu_t(\epsilon (B_1\cap B_2))=0. 
    \end{align}
    Now using \cref{p:rel-compact}, the above set of rescaled measures is relatively compact, so for every $t_n\to \infty$ there is $n_k \to \infty$ such that $\{\mu_{t_{n_k}}\}$ converges to some measure $\mu' \in \M_0.$ To prove the claim it is enough to show that any such $\mu'$ is equal to $\mu_{iid}^+$. This guarantees uniqueness of subsequential limits of $\mu_t$, which in turn implies convergence of $\mu_t$ to $\mu_{iid}^+.$\\
    
    Then by \cref{th:conv_criteria}, $\forall\; f \in \mathcal{C}_0,\; \int_{\R^2_0} f d\mu_t \longrightarrow \int_{\R^2_0} fd\mu' $ as $t \to \infty.$ Consider a closed BAFO rectangle $R_1$ and an open BAFO rectangle $R_2 \supset R_1$, both not touching the axes. More rigorously, if $A_x:=(0,\infty)\times\{0\}$ (the positive $X$-axis) and $A_y:=\{0\} \times (0,\infty)$ (the positive $Y$-axis), then $R_1\subset R_2 \subset \R_0^2\setminus \rbrac{A_x \cup A_y}$. Now, Urysohn's lemma guarantees us the existence of a continuous function $f$ such that $0\le f\le 1,\; f\equiv1  \text{ on } R_1 \text{ and supp}(f)=\overline{\{x: f(x)>0\}} \subset R_2. $ Then,
    \begin{align*}
        \int_{\R_0^2}fd\mu_t= \int_{R_2}fd\mu_t \leq \mu_t(R_2)
    \end{align*}
    Let $\{(a,y): y>0\} \text{ and } \{(x,b) : x>0\}$ be the left and bottom edge of $R_2$ respectively. Then $R_2 \subset (a \wedge b)(B_1 \cap B_2) \implies \mu_t(R_2)\leq \mu_t((a \wedge b) (B_1\cap B_2))$. Thus, by \eqref{eq: lim_boxes_zero},
    \begin{align*}
        &\lim_{t \to \infty} \int_{\R_0^2}f d\mu_t \leq \lim_{t\to \infty} \mu_t((a\wedge b)(B_1\cap B_2))=0\\
        & \implies \int_{\R_0^2}fd\mu'=0\\
        & \implies \int_{R_1}fd\mu'=0 \implies \mu'(R_1)=0
    \end{align*}
    The last step holds because $f$ is identically 1 on $R_1$. Hence, $\mu'$ is zero on any closed BAFO rectangle in $\R_0^2$ which does not touch the axes. Note that $\R_0^2 \setminus (A_x \cup A_y)$ is the countable union of such rectangles, so,
    \begin{align}\label{e: supp_limit_on_nonaxes}
        \mu'(\R_0^2\setminus(A_x \cup A_y))=\mu_{iid}^+\sqbrac{\R_0^2\setminus(A_x \cup A_y)}=0
    \end{align}
    \commb{To complete this proof, we need to show that $\forall $ Borel sets $E\subset \R_0^2$, $\mu'(E \cap A_x)=\mu_{iid}^+(E \cap A_x)$ and the analogous version with $A_y$. Since $\{[x,\infty)\}_{x\in\R_+}$ generates the Borel $\sigma$-algebra over $\R_+$, take an arbitrary $x \in \R_+$ and consider
    \begin{align*}
         \mu'([x,\infty)\times \{0\}) & =\mu'([x,\infty)\times \R) \quad (\text{as }\mu'\rbrac{\R_0^2 \setminus (A_x \cup A_y)}=0)\\
         & = \lim_{m\to\infty}\mu'([x-\delta_m,\infty)\times \R),\\
         & (\text{where }\delta_m \downarrow 0 \text{ is such that } \mu'\rbrac{\partial([x-\delta_m,\infty)\times \R)}=\mu'\rbrac{\{x-\delta_m\}\times \R}=0)\\
         &= \lim_{m\to \infty}\sqbrac{\lim_{k\to\infty} b(t_{n_k})\P(X\ge t_{n_k}(x-\delta_m))}\\
        &=\lim_{m\to\infty}c(x-\delta_m)^{-1}=cx^{-1}=\mu_{iid}^+\rbrac{[x,\infty)\times \{0\}}
    \end{align*}
    Note that $\delta_m\downarrow 0$ is guaranteed to exist as $x \mapsto \mu'\rbrac{[x,\infty) \times \R}$ is a monotone function with at most countably many discontinuities.}
    
    \commb{Similarly, $\forall x \in \R_+,\; \mu'(\{0\}\times [x,\infty))=cx^{-1}=\mu_{iid}^+(\{0\}\times [x,\infty))$. Combined with \eqref{e: supp_limit_on_nonaxes}, this proves $\mu'=\mu_{iid}^+$ for every subsequential limit of $\mu_t,$ which implies $\mu_t \Rightarrow \mu_{iid}^+ \text{ as } t \to \infty$ which proves the claim.}
\end{proof}
\begin{corollary}\label{c: non_neg_mrv}
    Let $X,Y$ be non-negative random variables in $ \RV_1(b,c_x)$ and $\RV_1(b,c_y)$ respectively. If they are also asymptotically independent, then, $(X,Y) \in \RV_1(b,\mu_{indep}^+)$ where $\mu_{indep}^+$ is the limit measure concentrated on the positive axes corresponding to the random vector comprised of independent positive $\RV_1(b,c_x) \text{ and }\RV_1(b,c_y)$ random variables. 
\end{corollary}
\begin{proof}
    Clearly, $X\in \RV_1(b,c_x), Y\in \RV_1(b,c_y) \implies X/c_x, Y/c_y \in \RV_1(b,1)$. Moreover, using the fact that $F_{X/c_x}^{-1}(p)=c_x^{-1}F_X^{-1}(p), \; F_{Y/c_y}^{-1}(p)=c_y^{-1}F_Y^{-1}(p) $,
    
    $$\lambda\rbrac{\frac{X}{c_x},\frac{Y}{c_y}}=\lim_{p \to 1^-}\P\rbrac{\frac{X}{c_x}>F_{X/c_x}^{-1}(p)\; \mid \frac{Y}{c_y}>F_{Y/c_y}^{-1}(p)}=\lambda(X,Y)=0$$
    
    Thus, $X/c_x \text{ and }Y/c_y$ are asymptotically independent too.\\
    By \cref{p:same_const_indep_implies_mrv}, 
    $$\rbrac{\frac{X}{c_x},\frac{Y}{c_y}} \in \RV_1(b,\mu_{iid}^+)$$
    Now note that, $\mu_{indep}^+$ is
    $$\mu_{indep}^+(E)=\mu_{c_x}(E_x)+\mu_{c_y}(E_y)\quad \forall \text{ Borel subsets } E \text{ of }\R_+^2\setminus\{\boldsymbol{0}\}$$
    where $E_x=\{x: (x, 0) \in E\} \text{ and }E_y=\{y: (0,y)\in E\}$, $d\mu_{c_{x}}=c_{x}u^{-2}du$ and $d\mu_{c_y}=c_yu^{-2}du$.
    To prove $(X,Y) \in \RV_1(b,\mu_{indep}^+)$, using Lemma 6.1 in \citet{resnick:1987}, it is enough to show that,
    $$\lim_{t \to \infty}b(t)\P\rbrac{\rbrac{\frac{X}{t},\frac{Y}{t}}\in[\boldsymbol{0},\boldsymbol{z}]^c}=\mu_{indep}^+([\boldsymbol{0},\boldsymbol{z}]^c)\quad \forall\;\boldsymbol{z}=(z_1,z_2)\in \R_+^2$$
    Indeed, 
    \begin{align*}
        & \lim_{t \to \infty}b(t)\P\rbrac{\rbrac{\frac{X}{t},\frac{Y}{t}}\in[\boldsymbol{0},\boldsymbol{z}]^c}\\
        & =\lim_{t \to \infty}b(t)\P\rbrac{\rbrac{\frac{X/c_x}{t},\frac{Y/c_y}{t}}\in([0,z_1/c_x]\times[0,z_2/c_y])^c}\\
        & = \mu_{iid}^+(([0,z_1/c_x]\times[0,z_2/c_y])^c)\\
        & = c_xz_1^{-1} + c_yz_2^{-1}\\
        & = \mu_{c_x}(([\boldsymbol{0},\boldsymbol{z}]^c)_x)+\mu_{c_y}(([\boldsymbol{0},\boldsymbol{z}]^c)_y)=\mu_{indep}^+([\boldsymbol{0},\boldsymbol{z}]^c)
    \end{align*}
    This proves the claim.
\end{proof}

\begin{proposition}\label{p:generalise_to_real_mrv}
    Let $X,Y$ be two real random variables with regularly varying upper and lower tails of index $1$, i.e. $\exists\; b(t) \to \infty$ and $c_X^\pm,c_Y^\pm>0$ such that $\forall x>0,$
    \begin{align}\label{e:RV_of_tails}
         \lim_{t \to \infty}b(t)\P(\pm X>tx)=c_X^\pm x^{-1}\ \  \text{ and } \ \ 
         \lim_{t \to \infty}b(t)\P(\pm Y>tx)=c_Y^\pm x^{-1}.
    \end{align}
    Suppose they are asymptotically independent in all tails, i.e., the following tail dependence coefficients are zero for all combinations of $\pm$:
    \begin{align}\label{e:all_tail_coeff_zero}
        \lambda(\pm X, \pm Y)=0.
    \end{align}
    Then, $(X,Y) \in \RV_1(b,\mu_{indep})$ where $\mu_{indep}$ is the limit measure concentrated on the axes corresponding to the random vector comprised of independent random variables with $\RV_1(b,c_X^{\pm})$ and 
    $\RV_1(b,c_Y^{\pm})$ tails, respectively. 
\end{proposition}

\begin{proof}
    Note that \eqref{e:all_tail_coeff_zero} implies $\lambda(X_{\pm},Y_{\pm})=0$, where $X_{+},Y_{+} \text{ and }X_{-},Y_{-}$ represent the positive and negative parts of X and Y respectively. Indeed, for large $p$,
    \begin{align*}
        \{-X>F_{-X}^{-1}(p)\}=\{X<-F_{-X}^{-1}(p)\}=\{X_{-}>F_{-X}^{-1}(p)\}
    \end{align*}
    as large $p$ implies $F_{-X}^{-1}(p)$ is positive. Note that due to assumption of regular variation of tails, support of $X$ extends to both $+\infty$ and $-\infty$ so $F_{-X}^{-1}(p)$ is guaranteed to be positive if we take $p$ sufficiently large.\\
    Now, for all $x>0, F_{-X}(x)=F_{X_-}(x)$. Thus, if $p$ is sufficiently large, $F_{X_-}^{-1}(p)=F_{-X}^{-1}(p)$. Thus,
    \begin{align*}
        \{-X>F_{-X}^{-1}(p)\}=\{X_{-}>F_{-X}^{-1}(p)\}=\{X_->F_{X_-}^{-1}(p)\}
    \end{align*}
    Similarly we can conclude that $\{Y>F_Y^{-1}(p)\}=\{Y_+>F_{Y_+}^{-1}(p)\}$ for large $p$.
    Therefore,
    \begin{align*}
        \lambda(X_-,Y_+)& =\lim_{p\to1-}\P(X_->F_{X_-}^{-1}(p) \mid Y_+>F_{Y_+}^{-1}(p))\\
        & =\lim_{p\to1-}\P(-X>F_{-X}^{-1}(p) \mid Y>F_{Y}^{-1}(p))=\lambda(-X,Y)=0
    \end{align*}
    Similarly,
    \begin{align*}
        \lambda(X_-,Y_-)=\lambda(X_+,Y_+)=\lambda(X_+,Y_-)=0
    \end{align*}
    Observe that \eqref{e:RV_of_tails} implies that $X_{\pm} \in \RV_1(b,c_X^\pm) \text{ and } Y_{\pm} \in \RV_1(b,c_Y^\pm)$. 
    Thus using \cref{c: non_neg_mrv}, $(X_{\pm},Y_{\pm}) \in \RV_1(b,\mu_{indep}^+)$.
    
    \commb{Now the rescaled measures $\mu_t(\cdot)=b(t)\P((X,Y)\in t\;\cdot )$ are relatively compact due to \cref{p:rel-compact}. Thus, $\forall t_n \to \infty, \exists n_k \to \infty$ such that $$\mu_{t_{n_k}} \overset{\M_0}{\Rightarrow}\mu' \in \M_0.$$
    Thus, it is enough to show that $\mu'=\mu_{indep}$ irrespective of the sequence and subsequence $t_n \text{ and }t_{n_k}$, respectively. Let $B_1=\{(x,y): \abs{x} \ge 1\} \text{ and }B_2=\{(x,y): \abs{y} \ge 1\}$. Also define $B_1^+=\{(x,y):x\ge 1\},\; B_2^+=\{(x,y): y \ge 1\}$ and $B_1^-=B_1\setminus B_1^+,\; B_2^-=B_2\setminus B_2^+$.}

    \commb{Consider $\mu'(\epsilon(B_1\cap B_2))$ as a function of $\epsilon > 0.$ This is a monotone function and hence has countably many discontinuities. Thus, there exists a sequence $\epsilon_m \downarrow 0$ such that $\mu'(\epsilon(B_1\cap B_2))$ is continuous at $\epsilon=\epsilon_m, \forall m \in \N \implies \mu'\sqbrac{\partial\rbrac{\epsilon_m(B_1\cap B_2)}}=0.$} 
    
    \commb{By the $\M_0$-convergence of $\mu_{t_{n_k}}$, we have,
    \begin{align*}
        & \mu'\sqbrac{\epsilon_m(B_1 \cap B_2)}=\lim_{k \to \infty} \mu_{t_{n_k}}\rbrac{\epsilon_m(B_1\cap B_2)}\\
        & =\lim_{k \to \infty}b(t_{n_k})\P\rbrac{(X,Y) \in \epsilon_m t_{n_k}(B_1^+ \cap B_2^+)}+\lim_{k \to \infty}b(t_{n_k})\P\rbrac{(X,Y) \in \epsilon_m t_{n_k}(B_1^+ \cap B_2^-)}\\
        & +\lim_{k \to \infty}b(t_{n_k})\P\rbrac{(X,Y) \in \epsilon_m t_{n_k}(B_1^- \cap B_2^+)}+\lim_{k \to \infty}b(t_{n_k})\P\rbrac{(X,Y) \in \epsilon_m t_{n_k}(B_1^- \cap B_2^-)}\\
        & =\lim_{k \to \infty}b(t_{n_k})\P\rbrac{(X_+,Y_+) \in \epsilon_m t_{n_k}(B_1^+ \cap B_2^+)}+\lim_{k \to \infty}b(t_{n_k})\P\rbrac{(X_+,Y_-) \in \epsilon_m t_{n_k}(B_1^+ \cap B_2^-)}\\
        & +\lim_{k \to \infty}b(t_{n_k})\P\rbrac{(X_-,Y_+) \in \epsilon_m t_{n_k}(B_1^- \cap B_2^+)}+\lim_{k \to \infty}b(t_{n_k})\P\rbrac{(X_-,Y_-) \in \epsilon_m t_{n_k}(B_1^- \cap B_2^-)}\\
        & = \mu^{+,+}_{indep}\rbrac{\epsilon_m(B_1^+\cap B_2^+)} + \mu^{+,-}_{indep}\rbrac{\epsilon_m(B_1^+\cap B_2^-)}\\
        & + \mu^{-,+}_{indep}\rbrac{\epsilon_m(B_1^-\cap B_2^+)}
        + \mu^{-,-}_{indep}\rbrac{\epsilon_m(B_1^-\cap B_2^-)}\\
        & = 0,
    \end{align*}
    as $\epsilon_m(B_1^\pm \cap B_2^\pm)$ is a continuity set for all $\mu^{\pm,\pm}_{indep}.$ Since $\mu'\rbrac{\epsilon_m(B_1\cap B_2)}=0$ for all $m$, taking $m \to \infty$ we get that $\mu'\sqbrac{\R_0^2\setminus \rbrac{A_x \cup A_y}}=0$
    where $A_x, A_y$ denote the $X$- and $Y$-axes, respectively.}
    
    \commb{Next, consider the following for any $x\in \R_+,$
    \begin{align*}
        \mu'\rbrac{[x,\infty) \times \{0\}}& =\mu'\rbrac{[x,\infty) \times \R} \quad \rbrac{\text{since }\mu'\sqbrac{\R_0^2\setminus(A_x \cup A_y)}=0}\\
        &= \lim_{m \to \infty} \mu' \rbrac{[x-\delta_m, \infty) \times \R},\; \text{where }\delta_m \downarrow 0 \text{ is such that } \mu'\rbrac{\{x-\delta_m\} \times \R}=0, \forall m,\\
        & = \lim_{m \to \infty}\sqbrac{\lim_{k \to \infty} \mu_{t_{n_k}}\rbrac{[x-\delta_m,\infty) \times \R}}\\
        & =\lim_{m \to \infty}\sqbrac{\lim_{k \to \infty} b(t_{n_k})\P\cbrac{X\ge t_{n_k}(x-\delta_m)}}\\
        &=\lim_{m\to \infty} \sqbrac{c_X^+ (x-\delta_m)^{-1}}=c_X^+ x^{-1}=\mu_{indep}\rbrac{[x,\infty) \times \{0\}}
    \end{align*}
    Similarly, $\forall x\in \R_+,$
    \begin{align*}
        & \mu'\rbrac{(-\infty,-x] \times \{0\}}=\mu_{indep}\rbrac{(-\infty,-x] \times \{0\}},\\
        & \mu'\rbrac{\{0\}\times(-\infty,-x] }=\mu_{indep}\rbrac{\{0\} \times(-\infty,-x]},\\
        & \mu'\rbrac{\{0\} \times[x,\infty)}=\mu_{indep}\rbrac{\{0\} \times[x,\infty)}
    \end{align*}
    Combined with the fact that $\mu'\sqbrac{\R_0^2\setminus \rbrac{A_x \cup A_y}}=\mu_{indep}\sqbrac{\R_0^2\setminus \rbrac{A_x \cup A_y}}=0$, this proves that $\mu'=\mu_{indep}$. Hence,
    \begin{equation*}
        \mu_t \overset{\M_0}{\Rightarrow}\mu_{indep} \implies (X,Y) \in \RV_1(b,\mu_{indep}).
    \end{equation*}
    }
\end{proof}

\begin{theorem}\label{th:asym indep & heavy tails implies mrv}
    Let $X=(X_i)_{i=1}^d$ be a random vector whose marginals have regularly varying distributions with index $1$, i.e., $\exists \text{ a monotone increasing function }b(t) \to \infty \text{ and } c_{\pm}(i)>0$ such that
    $$\lim_{t \to \infty}b(t)\P\rbrac{\pm X_i>tx}=c_{\pm}(i)x^{-1} \quad \forall x>0  \text{ and } \forall i=1,\ldots,d.$$
    If  $\; \forall\; 1 \leq i \neq j \leq d$ we have
    $$\lambda(\pm X_i, \pm X_j)=0,$$
    then, $X \in \RV_1(b,\mu_{indep}^{(d)})$, where $\mu_{indep}^{(d)}$ is the same as that in \cref{p:generalise_to_real_mrv} but in $d\in \N$ dimensions. 
\end{theorem}

\begin{proof}
\commb{The proof essentially follows the same strategy as \cref{p:generalise_to_real_mrv}. As a result we outline the steps instead of going through all the calculations.}
\begin{enumerate}
    \item \commb{\textbf{Step 1}: Due to \cref{p:rel-compact}, the rescaled measures $\mu_t(\cdot)=b(t)\P\rbrac{X/t \in  \cdot}$ are relatively compact. Thus, $\forall t_n \to \infty$, there exists a subsequence $n_k \to \infty$ and some $\mu'\in \M_0$ such that $$\mu_{t_{n_k}} \overset{\M_0}{\Rightarrow}\mu'$$
    We show $\mu'$ is invariant of the sequence or subsequence and $\mu'(E)=\mu_{indep}^{(d)}(E), \; \forall$ Borel sets $E\subset \R_0^d.$ As usual, we partition the space into $\R_0^d\setminus(\cup_{i=1}^d A_i), \{A_i\}_{i=1}^d$ where $A_i$ represents the $i-$th axis.}
    \item \commb{\textbf{Step 2}: Define 
    \begin{equation*}
    B_j^{(d)}:=B_j^+ \sqcup B_j^-=\{x\in \R^d: x_j \ge 1\} \sqcup \{x\in \R^d: x_j \le -1\}
    \end{equation*}
    for $j=1,\ldots,d.$ Using the monotonicity of $\epsilon>0 \mapsto \mu'\left(\epsilon\cdot(\bigcap_{i} B_i^{(d)})\right)$, we get that $\exists \epsilon_m \downarrow 0 \text{ such that } \mu'\left[\partial\rbrac{\epsilon_m\cdot(\bigcap_{i}B_i^{(d)})}\right]=0$ for all $ m$.
    \[
        \implies \mu'\sqbrac{\epsilon_m\rbrac{\bigcap_{i=1}^dB_i^{(d)}}}=\lim_{k \to \infty} \mu_{t_{n_k}}\sqbrac{\epsilon_m\rbrac{\bigcap_{i=1}^dB_i^{(d)}}}
    \]}
    \item \commb{\textbf{Step 3.} Use \cref{p:generalise_to_real_mrv} to prove support of $\mu'$ is contained in $\cup_{i}A_i$:
    \begin{align*}
        \lim_{k \to \infty} \mu_{t_{n_k}}\sqbrac{\epsilon_m\rbrac{\bigcap_{i=1}^dB_i^{(d)}}}& =\lim_{k \to \infty} b\rbrac{t_{n_k}}\P\rbrac{\forall i=1,\ldots,d, |X_i|> t_{n_k}\epsilon_m}\\ 
        & \le \lim_{k \to \infty} b\rbrac{t_{n_k}}\P \rbrac{(X_1,X_2)\in \epsilon_m t_{n_k}(B_1^{(2)} \cap B_2^{(2)})}
    \end{align*}
    Now $(X_1,X_2)\in \RV_1(b,\mu_{indep}^{(2)})$ by \cref{p:generalise_to_real_mrv}. Since $\epsilon_m(B_1^{(2)}\cap B_2^{(2)})$ is a continuity set of $\mu_{indep}^{(2)}$, the above equation implies:
    \begin{align*}
        \lim_{k \to \infty} \mu_{t_{n_k}}\sqbrac{\epsilon_m\rbrac{\bigcap_{i=1}^dB_i^{(d)}}} & \le \lim_{k \to \infty} b\rbrac{t_{n_k}}\P \rbrac{(X_1,X_2)\in \epsilon_m t_{n_k}(B_1^{(2)} \cap B_2^{(2)})}\\
        & =\mu_{indep}^{(2)}\rbrac{\epsilon_m(B_1^{(2)}\cap B_2^{(2)})}=0
    \end{align*}
    Thus, 
    \begin{align*}
        &\forall m,\;\mu'\sqbrac{\epsilon_m\rbrac{\bigcap_{i=1}^dB_i^{(d)}}}=0\\
        & \implies\mu'\sqbrac{\R^d_0 \setminus \bigcup_{i=1}^d A_i}=\lim_{m\to \infty} \mu'\sqbrac{\epsilon_m\rbrac{\bigcap_{i=1}^dB_i^{(d)}}}=0=\mu_{indep}^{(d)}\sqbrac{\R^d_0 \setminus \bigcup_{i=1}^d A_i}
    \end{align*}}
    
    \item \commb{\textbf{Step 4:} Prove that mass on the axes put by $\mu'$ matches $\mu_{indep}^{(d)}$: $\forall x \in(0,\infty),$
    \begin{align*}
        \mu'\sqbrac{[x,\infty) \times \{0\}^{d-1}}& =\mu'\sqbrac{[x,\infty)\times \R^{d-1}}, \quad(\text{since } \mu'\sqbrac{\R^d_0 \setminus \cup_{i=1}^d A_i}=0)\\
        & =\lim_{m \to \infty} \mu'\sqbrac{[x-\delta_m,\infty)\times \R^{d-1}}\\
        & \rbrac{\delta_m \downarrow 0 \text{ such that } \mu'\sqbrac{\partial([x-\delta_m,\infty)\times \R^{d-1})}=0}\\
        & = \lim_{m \to \infty} \sqbrac{\lim_{k \to \infty} b(t_{n_k})\P\rbrac{X_1\ge t_{n_k}(x-\delta_m)}}\\
        &= \lim_{m \to \infty} c_+(1) (x-\delta_m)^{-1}=c_+(1)x^{-1}=\mu_{indep}^{(d)}\rbrac{[x,\infty)\times \{0\}^{d-1}}
    \end{align*}
    The same holds for sets of the form $\{0\}^{i-1} \times [x,\infty)\times \{0\}^{d-i}$ and $\{0\}^{i-1} \times (-\infty,-x]\times \{0\}^{d-i}$ for all $x \in \R_+$ and $i \ge 1.$ Thus $\mu'=\mu_{indep}^{(d)}$ on the axes as well which proves the theorem.}
\end{enumerate}
\end{proof}

\subsection{Additional examples of multivariate regular variation}\label{sec: examples supp}

\begin{example}[max-linear heavy-tailed factor models] Let the $Z_j$'s and the matrix $A$ be as in 
\cref{ex:linear_factor_model}.  Consider the model
$$
X = \bigvee_{j=1}^p a_j Z_j =: A\ovee Z,
$$
where $\bigvee$ denotes component-wise maxima of the vectors $a_j Z_j$ and the $a_j$'s are the columns of the
matrix $A$.  Thus $X$ is obtained by replacing the `$+$' operation in the definition of matrix multiplication by a
maximum.  Interestingly, the single large jump heuristic here entails that $X\in \RV_\beta(\R^d, b(t)=t^{\beta}, \mu)$, where
$\mu$ is {\em the same} as for the linear model in \cref{ex:linear_factor_model}.  Consequently, the corresponding
angular measure associated with $\mu$ is \eqref{e:ex:linear_factor_model}.
\end{example}
The following two examples illustrate a small part of the limit theorems for regularly varying 
random vectors.  Specifically, if one considers centered and rescaled component-wise sums (or maxima, respectively), 
the corresponding limit random vectors will have sum-stable (or max-stable, respectively) distributions.  Except in 
the Gaussian case, these sum-stable (max-stable, respectively) laws are {\em multivariate regularly varying}.

\begin{example}[multivariate max-stable distributions]\label{ex:max-stable}
Fix $\beta>0$ and let $\mu$ be an arbitrary non-zero 
Borel measure on $\R^d$, supported on $[0,\infty)^d\setminus\{0\}$ and such that
\begin{equation}\label{e:ex:max-stable-scaling}
 \mu(t\cdot A) = t^{-\beta} \mu(A) <\infty,
\end{equation}
for all $t>0$ and Borel $A \subset \R^d$ that are bounded away from $0$.

Then,
\begin{equation}\label{e:ex:max-stable}
F(x):= \exp\{ -\mu(\R_+^d \setminus [0,x])\},\ \ x\in (0,\infty)^d
\end{equation}
defines a valid cumulative distribution function of a random vector $X$, which is {\em multivariate regularly varying}
\citep[see e.g. Chapter 5 in][]{resnick:1987}.  More precisely, we have $X\in \RV_\beta(\R^d, b(t)=t^{\beta},\mu)$ and
in fact, the random vector $X$ is {\em max-stable}.  That is, for all integer $n\ge 1$, 
$$
\bigvee_{i=1}^n X(i) \stackrel{d}{=} n^{1/\beta} X,
$$
where the $X(i)$'s are independent copies of $X$ and `$\vee$' denotes the component-wise maximum operation.

The scaling property \eqref{e:ex:max-stable-scaling} implies that for any fixed norm $\|\cdot\|$ in $\R^d$, we can rewrite
$$
F(x) = \Pb[ X\le x] = 
 \exp\Big\{ -\int_{S_+} \Big( \max_{i=1,\ldots,d} \frac{\theta_i}{x_i}\Big)^\beta H (d\theta) \Big\},\ \ 
 x\in (0,\infty)^d,
$$
where $S_+:= S_{\|\cdot\|} \cap [0,\infty)^d$ is the positive part of the unit sphere in the chosen norm $\|\cdot\|$ and $H$ is some measure on it.

The angular measure $\sigma$ associated with the exponent measure $\mu$ is a normalized version of $H$:
$$
\sigma(A) = \frac{H(A)}{H(S_+)},\ \ \ A\subset S_+.
$$

Upon centering and transformation of the marginal distributions, the above class of multivariate max-stable laws
represent the entire class of {\em extreme value distributions}.  That is, the distributions arising in the limit
of centered and rescaled maxima of iid random vectors. For more details \citep[see, e.g.,][]{resnick:1987,beirlant2004statistics,resnick:2007}.
\end{example}
\begin{remark}\label{r:max-id} The powerful Poisson random measure perspective \citep[see e.g.][]{resnick:1987,resnick:2007} 
leads to a quick proof of the fact that Relation \eqref{e:ex:max-stable} yields a valid distribution function. 
Indeed, take $\Pi = \{\xi_i,\ i\in \N\}$ to be a Poisson point process on $\R_+^d = [0,\infty)^d$
with mean measure $\mu$ and define
$$
X := \bigvee_{i\in \N} \xi_i.
$$
Then, for all $x\in (0,\infty)^d$, we have
\begin{equation}\label{e:r:max-id}
\Pb[ X\le x] = \Pb[\Pi ([0,x]^c) = 0 ] = \exp\{ -\mu([0,x]^c)\}, 
\end{equation}
where the last equality follows from the fact that $\Pi(A)\sim {\rm Poisson}(\mu(A))$, for every Borel set 
$A\subset \R_+^d$.  This is precisely \eqref{e:ex:max-stable}.

Notice that this argument does not depend on the scaling property \eqref{e:ex:max-stable-scaling}. 
The general family of multivariate distributions as in \eqref{e:r:max-id} is known as {\em max-infinitely divisible}
distributions and many of them are multivariate regularly varying \citep[see e.g. Chapter 5 in ][]{resnick:1987}.
\end{remark}

\begin{example}[stable non-Gaussian distributions] \label{ex:SaS}
Recall that a random vector $X$ in $\R^d$ is said to be sum-stable, if for all 
positive constants $a', a''$ there exist positive $a$ and a vector $b\in\R^d$ such that
$$
a' X' + a'' X'' \stackrel{d}{=} a X + b,
$$
where the $X'$ and $X''$ are independent copies of $X$ \citep[Definition 2.1.1 on page 57 in][]{samorodnitsky:taqqu:1994book}.

We focus on the rich family of {\em symmetric} stable non-Gaussian distributions.
Fix an arbitrary norm $\|\cdot\|$ in $\R^d$.  It is a well-known non-trivial fact that every 
symmetric non-Gaussian sum-stable random vector $X$ has a characteristic function of the form:
\begin{equation}\label{e:ex:SaS-ch-fun}
 \E[ e^{i X^{\T} u} ] =\exp\Big\{ -\int_{S_{\|\cdot\|}} |\langle u,\theta\rangle|^\beta \Gamma(d\theta) \Big\},\ \ \mbox{
  where } 0<\beta<2
\end{equation}
\citep[see, e.g., Theorem 2.4.3 in][]{samorodnitsky:taqqu:1994book}, for some $\Gamma$ -- a finite 
symmetric measure on the unit sphere $S_{\|\cdot\|}$ in the chosen norm $\|\cdot\|$. (Note that $\Gamma$ depends on the
choice of the norm.) Conversely, every finite symmetric measure $\Gamma$ on $S_{\|\cdot\|}$ yields a characteristic function 
of a symmetric $\beta$-stable (S$\beta$S) random vector $X$ as above.  

The case $\beta=2$ yields a Gaussian random vector. Interestingly, when $0<\beta<2$, 
the S$\beta$S random vector $X$ is {\em multivariate regularly varying}
with exponent $\beta$ and angular measure
$$
\sigma(A) = \frac{\Gamma(A)}{\Gamma(S_{\|\cdot\|})},\ \ A\subset S_{\|\cdot\|}.
$$
Specifically, Theorem 4.4.8 on page 197 in \citet{samorodnitsky:taqqu:1994book} implies that 
$X \in \RV_\beta(\R^d, b(t)= t^\beta,\mu)$, where 
$\mu (B_{\|\cdot\|}(0,1)^c) = C_\beta  \Gamma(S_{\|\cdot\|})$  with
$$
C_\beta = \left\{\begin{array}{ll}
 \frac{1-\beta}{\Gamma(2-\beta)\cos(\pi\beta/2)} &,\ \beta\not=1\\
 2/\pi &,\ \beta =1
\end{array}\right.
$$
\citep[cf (1.2.9) on page 17 in][]{samorodnitsky:taqqu:1994book}.
\end{example}

\begin{remark}[Aside on notation]\label{r:alpha-vs-beta} Since $\alpha$ is reserved for the level of the type-I error here, we 
use $\beta$ to denote the tail exponent.  In the literature on non-Gaussian sum-stable 
distributions \citep[see, e.g.][]{samorodnitsky:taqqu:1994book}, 
$\alpha$ stands for the tail-exponent (stability index), while $\beta$ denotes the skewness parameter.
\end{remark}

The following example provides an alternative and analytically more convenient 
representation to the class of symmetric $\beta$-stable random vectors as discussed in \cref{ex:SaS}.  Moreover, for the class of models corresponding to $\beta=1$, the exact, non-asymptotic, calibration properties of the Cauchy combination test can be thoroughly understood.  

For further details on non-Gaussian stable random vectors and processes,
we refer the reader to the classical monograph of \citet{samorodnitsky:taqqu:1994book}.  We will only 
review some basic notation and facts here.

\begin{example}[Multivariate S$1$S laws] \label{ex:S1S}
We begin with a rigorous definition of symmetric $\beta$-stable variables.
\begin{definition}[Symmetric $\beta$-stable (S$\beta$S)]
    Let $0<\beta\le 2$. A random
 variable $\xi$ is said to have a symmetric $\beta$-stable (S$\beta$S) distribution if
 $$
 \varphi_\xi(t) = \E[ e^{it \xi}] = e^{-\sigma_\xi^\beta |t|^\beta},\ \ \ t\in \R,
 $$
 for some scale coefficient $\sigma_\xi>0$.  We shall denote the scale coefficient 
 $\sigma_\xi$ of $\xi$ as $\|\xi\|_\beta$. (Not to be confused with a norm.)  
\end{definition}

If $0<\beta <2,$
we have that the S$\beta$S random variables are non-Gaussian and 
{\em heavy-tailed} in the sense that
\begin{equation}\label{e:tail-SaS}
\Pb [ \xi > t ] \sim c_\beta \frac{\|\xi\|_\beta^\beta}{t^\beta},\ \ \mbox{ as }t\to\infty,
\end{equation}
for some constant $c_\beta$.
\begin{definition}[Multivariate S$\beta$S]
     A random vector $X = (X_i)_{i=1}^d$ is said to be multivariate S$\beta$S (or just S$\beta$S)
 if for all
 $a_j\in \R$, we have that $\sum_{j=1}^d a_j X_j$ is S$\beta$S. 
\end{definition}

 This definition is equivalent to the one discussed in \cref{ex:SaS} for the case of symmetric random 
 vectors. The joint characteristic function of S$\beta$S random vectors given in \eqref{e:ex:SaS-ch-fun},
 can be equivalently expressed using the following fact \citep[see Chapter 3 in][]{samorodnitsky:taqqu:1994book}. 

 A random vector $X$ is S$\beta$S if and only if there exist $f_j \in L^\beta([0,1])$ such that
 \begin{align*}
 \varphi_X(t_1,\ldots,t_d) = \E e^{i \sum_{j=1}^d t_j X_j}
 = \exp\Big\{ - \int_{[0,1]} \Big| \sum_{j=1}^d t_j f_j(u)\Big|^\beta du \Big\}
 \end{align*}
 for all $t_j\in \R,\ j=1,\ldots,d$.  This means in particular that the scale
 coefficient of the S$\beta$S random variable $\xi := \sum_{j=1}^d t_j X_j$ equals
 \begin{align}\label{e:SaS}
\Big\|\sum_{j=1}^d t_j X_j\Big\|_\beta 
&= \Big(\int_{[0,1]} \Big| \sum_{j=1}^d t_j f_j(u)\Big|^\beta du\Big)^{1/\beta}
\end{align}
Conversely, every choice of $f_j \in L^\beta([0,1]),\ j=1,\ldots,d$ yields
a joint characteristic function of an S$\beta$S random vector as above.

As discussed in \cref{ex:SaS}, all non-Gaussian S$\beta$S vectors are multivariate regularly varying as well.  Their
angular measure can be expressed as:
$$
\sigma(\cdot)=\frac{\int_{0}^1 \mathbb{I}[f(u)/\norm{f(u)} \in \cdot]\norm{f(u)}^{\beta}du}{\int_{0}^1\norm{f(u)}^\beta du},
$$
where $f(u)$ denotes the vector-valued function $(f_j(u))_{j=1}^d,\ u\in [0,1]$ and $\|\cdot\|$ is the corresponding norm 
associated with the angular measure.
In the case of $\beta=1$, the sum-stability of S$\beta$S vectors allows one to directly express the calibration properties of the Cauchy combination tests, as shown in the following corollary.
\end{example}

\begin{corollary}\label{c: s1s calib}
    Let $P_i,\ i=1,\ldots,d$ be Uniform$(0,1)$ distributed random variables and let $X_i:= \tan\rbrac{\pi\rbrac{\frac{1}{2}-P_i}} \sim$ standard Cauchy. Say $X:=(X_i)_{i=1}^d$ is multivariate S1S and $(w_i)_{i=1}^d$ are non-negative weights which sum to 1. Then, the Cauchy combination test defined with these weights is asymptotically honest, i.e.,
    \begin{equation*}
        \lim_{t\to\infty}\frac{\Pb(\sum_{i=1}^dw_iX_i>t)}{\Pb(X_1>t)}\leq 1
    \end{equation*}    
    Moreover, equality holds above if and only if for all $i,j$ with $w_iw_j>0$ we have $f_i(u)f_j(u)\ge 0$ for a.e. $u\in[0,1]$. In this case, the Cauchy combination test is exactly calibrated at all levels, not just asymptotically.
\end{corollary}
\begin{proof}
    For $\beta=1$ (S1S), any linear combination is Cauchy. Here, we assume that the coordinates have
unit scale,
\[
\|X_j\|_1=\int_0^1 |f_j(u)|\,du = 1 ,\qquad j=1,\dots,d.
\]
For weights $w_j\ge 0$ with $\sum_{j=1}^d w_j=1$, the Cauchy combination test considers
\[
T=\sum_{j=1}^d w_j X_j .
\]
Then, $T$ is Cauchy with scale
\[
\|T\|_1=\int_0^1 \Big|\sum_{j=1}^d w_j f_j(u)\Big|\,du ,
\]
and, in view of \eqref{e:tail-SaS}, the tail ratio satisfies
\begin{equation}\label{e: lim ratio scale of S1S}
\lim_{t\to\infty}\frac{\Pb(T>t)}{\Pb(X_1>t)}=\|T\|_1 .
\end{equation}
By convexity (triangle inequality),
\[
\|T\|_1 \le \sum_{j=1}^d |w_j| \int_0^1 |f_j(u)|\,du = 1 ,
\]
so rejecting for $T>F^{-1}_{X_1}(1-\alpha)$ yields an asymptotic type-I error $\le\alpha$.

For the equality condition, without loss of generality assume that $w_i>0\;\forall i.$ If not, the following argument directly applies to the subset with strictly positive weights. If the spectral functions are \emph{spectrally positive}, i.e.
\[
f_i(u)f_j(u)\ge 0 \quad \text{for a.e. }u\in[0,1]\ \text{and all }i,j,
\]
 then,
\[
\|T\|_1=\int_0^1 \Big|\sum_{j=1}^d w_j f_j(u)\Big|\,du
=\sum_{j=1}^d w_j \int_0^1 |f_j(u)|\,du
=\sum_{j=1}^d w_j = 1 .
\]
Hence $T$ is \emph{standard Cauchy}, and for every level $\alpha\in(0,1)$,
\[
\Pb\!\big(T>F^{-1}_{X_1}(1-\alpha)\big)=\alpha,
\]
i.e. the Cauchy combination test is \emph{exactly calibrated at all levels}. Thus, it is also asymptotically calibrated. For the other direction, note that equality in \eqref{e: lim ratio scale of S1S} holds iff
\begin{equation*}
    \left\vert \sum_{i=1}^d w_if_i(u) \right\vert=\sum_{i=1}^d w_i \abs{f_i(u)} \text{ for a.e. }u\in[0,1]
\end{equation*}
which implies spectral positivity.
\end{proof}
\begin{remark}
    Spectral positivity of the functions implies that the exponent measure is supported on the positive and negative orthants. As a result, \cref{c:CCT} applies and we arrive at asymptotic calibration for this copula. However, as we proved, calibration is not just asymptotic, but exact for this case.
\end{remark}

\section{Proofs} \label{sec:supp-proofs}
 \subsection{Proof of \cref{p:general}}\label{sec: proof-general}

  \begin{proof} \commb{Let $w_1,\dots,w_d$ be fixed. Consider the following non-negative, continuous, $1$-positively-homogeneous functions
  $$
  h(x) =\Big(\sum_{i=1}^d w_i x_i \Big)_+ \mbox{ and }\ \ h_i(x):= (x_i)_+,\ i=1,\ldots,d,
  $$
For every $t>0$, using the fact that $x > t$ if and only if $(x)_+ > t$, it holds that 
  $$
  \Pb[ T_w(X) >t] = \Pb[ h(X)>t] \ \mbox{ and }\ \ \Pb[X_i>t] = \Pb[ h_i(X) >t ],\; i=1,\dots,d.
  $$
  \cref{l:homogeneous} implies that as $t \to +\infty$,
\[b(t) \Pb[T_w(X) >t] \to c_\mu \E\{h(\Theta)^\beta\} \ \mbox{ and } \ \ b(t)\Pb[X_i>t]\to c_\mu \E[ h_i(\Theta)^\beta],\, i=1,\dots,d.\]
Assumption \eqref{e:p-general} entails
  $\E\{ h_i(\Theta)^\beta\} = \E\{ (\Theta_i)_+^\beta\} = c/c_{\mu}>0$ for $i=1,\dots,d$.
  Further, taking the ratio of the limits in the display above, we obtain
  \eqref{e:p-general-limit}. The fact that the calibration constant is at most $1$ when $\beta\ge 1$ follows from Jensen's inequality for the function $x \mapsto (x)_+^\beta$.} 
  \end{proof}
\subsection{Proof of \cref{c:CCT}} \label{sec:proof-cauchy}
\begin{proof}
We complete the proof for the case of equality.
    
     If $\text{supp } \sigma \subseteq \R_-^d \cup \R_+^d$,
     $$(\Theta_j)_+=0\; \forall\; j \text{ or } (\Theta_j)_+=\Theta_j\; \forall\; j, \quad \text{$\sigma$-a.s.}$$
     In both the above cases,

     $$\rbrac{\sum_{i=1}^dw_i\Theta_i}_+=0=\sum_{i=1}^dw_i(\Theta_i)_+ \text{ or } \rbrac{\sum_{i=1}^dw_i\Theta_i}_+=\sum_{i=1}^dw_i\Theta_i=\sum_{i=1}^dw_i(\Theta_i)_+, \quad \text{$\sigma$-a.s.}$$
     Thus,
     $$\E\sqbrac{\rbrac{\sum_{i=1}^dw_i\Theta_i}_+}=\sum_{i=1}^dw_i\E\{(\Theta_i)_+\}=\E\{(\Theta_1)_+\}$$
     By \eqref{e: +ve part calibration limit}, 
     \begin{align*}
         \lim_{t\to\infty} \frac{\P[T_w(X)>t]}{\P[X_1>t]} = \frac{1}{\E (\Theta_1)_+}  \E \Big ( \sum_{j=1}^d w_j \Theta_j\Big)_+ = 1
     \end{align*}
     and (asymptotic) calibration holds.\\
     Now, for the converse to hold, one can easily see that Jensen's inequality used in proving honesty, needs to hold with equality almost surely, i.e.,
     \begin{align}
         & \lim_{t\to\infty} \frac{\P[T_w(X)>t]}{\P[X_1>t]} = \frac{1}{\E (\Theta_1)_+}  \E \Big ( \sum_{j=1}^d w_j \Theta_j\Big)_+=\frac{1}{\E\rbrac{ \sum_{j=1}^dw_j(\Theta_j)_+}}  \E \Big ( \sum_{j=1}^d w_j \Theta_j\Big)_+ = 1\notag\\
         & \implies \E\rbrac{\Big ( \sum_{j=1}^d w_j \Theta_j\Big)_+ - \sum_{j=1}^dw_j(\Theta_j)_+}= 0\notag \\
         & \implies \Big ( \sum_{j=1}^d w_j \Theta_j\Big)_+ = \sum_{j=1}^dw_j(\Theta_j)_+, \quad \text{$\sigma$-a.s.} \label{e : equality in jensen holds a.s}
     \end{align}
     as the random variable inside the expectation is always non-negative due to Jensen's. This claim can be proved using the following general result: Let $f:\R\to \R$ be a convex function, and suppose there exist $x_1,\ldots,x_m\in\R$ and weights $w_1,\ldots,w_m>0$ with $\sum_{i=1}^m w_i=1$ for which
     $$f\rbrac{\sum_{i=1}^m w_ix_i}=\sum_{i=1}^m w_i f(x_i)$$
     i.e., equality in Jensen's holds. Then $f$ \emph{must} be affine over the convex hull of $\{x_i\}$. In our case, $f(x)=x_+$ is affine only in $\R_+$ and $\R_-$. Thus, equality in Jensen's implies $\text{Conv}(\Theta_i: i=1,\ldots,d) \subseteq \R_+ \cup \R_- \implies \Theta_i \in \R_+\; \forall i$ or $\Theta_i \in \R_-\; \forall i$. However, for completeness, we also include an elementary proof below.\\ 
     Take any $\theta=(\theta_1,\ldots,\theta_d)$. Let $\theta_k=\min_i \theta_i$ and $\theta_l=\max_i \theta_i$, which we may assume is positive. Then,
     \begin{align*}
         &  \sum_{j=1}^d w_j \theta_j=  w^* \theta_k + (1-w^*)\theta_l\\
         & \text{ where }w^*=\sum_{j=1}^dw_j\rbrac{\frac{\theta_j-\theta_l}{\theta_k-\theta_l}}\in[0,1]
     \end{align*}
     Now, since we assume $w_j>0$ for all $j$, there exists $\alpha^*\in (0,1]$ such that
     \begin{align}
         \alpha^*(\theta_l)_+=\sum_{j=1}^d w_j (\theta_j)_+ \label{e : strict convexity}
     \end{align}
     Thus, we have
     \begin{align}
         & \Big ( \sum_{j=1}^d w_j \theta_j\Big)_+= \sum_{j=1}^d w_j (\theta_j)_+ \label{e: eq jensen holds}\\
         & \implies( w^* \theta_k + (1-w^*)\theta_l)_+ =\alpha^*(\theta_l)_+>0\notag\\
         &\implies \alpha^*=w^*(\theta_k/\theta_l-1)+1=\sum_{j=1}^dw_j(\theta_j-\theta_l)/\theta_l+1=\sum_{j=1}^dw_j\theta_j/\theta_l\notag\\
         & \implies \sum_{j =1 }^d w_j(\theta_j)_+/\theta_l=\sum_{j=1}^dw_j\theta_j/\theta_l\notag \\
         & \implies \theta_j \geq 0\ \forall\, j \notag
     \end{align}
     As a result, if \eqref{e: eq jensen holds} holds, $\exists\theta_i>0\implies \theta\in\R_+^d$. Therefore,
     \begin{align}
         & \Big ( \sum_{j=1}^d w_j \theta_j\Big)_+= \sum_{j=1}^d w_j (\theta_j)_+\notag \\
         & \implies \theta \in \R_+^d \cup\R_-^d \notag
     \end{align}
     This means, \eqref{e : equality in jensen holds a.s} implies 
     \begin{align}
         \Theta \in \R^d_+ \cup \R^d_-, \quad \text{$\sigma$-a.s.}
     \end{align}
     which proves the only if direction and hence completes the proof.
 \end{proof}

\subsection{Proof of \cref{l:eq_univ_calib}} \label{sec:supp pf of univ calib lemma}

\begin{proof}
\commb{\textbf{Only if direction:}}
If $\sigma$ is some probability measure on $\Delta^{d-1}=\{ (w_i)_{i=1}^d\, :\, w_i\ge 0,\ \sum_i w_i = 1\}$ such that
\begin{equation}\label{e: supp_marginal_cond_sigma}
    \E_{\Theta\sim\sigma}[\Theta_i]=\frac{1}{d},\; \forall i=1,\ldots, d,
\end{equation}
take $W\sim \sigma$ on $\Delta^{d-1}$ and $Y/d\sim$ unit Pareto independent of $W$. Then, $X=YW$ is multivariate regularly varying with index 1 by Breiman's lemma (\cref{lem:Breiman}). Moreover, $\norm{X}_1=Y$ and $X/\norm{X}_1=W\sim \sigma$. By \cref{l:homogeneous}, as $t \longrightarrow \infty,$
\begin{align*}
    & t\P\rbrac{\norm{X}_1>t}\longrightarrow c_{\mu_X} \text{ and, }\forall i=1,\ldots,d,\; t\P\rbrac{X_i>t} \longrightarrow c_{\mu_X}/d\\
    \implies & t\P\rbrac{Y>t}\longrightarrow c_{\mu_X} \text{ and, }\forall i=1,\ldots,d,\; t\P\rbrac{X_i>t} \longrightarrow c_{\mu_X}/d\\
    \implies & c_{\mu_X}=d \text{ and hence } \forall i=1,\ldots,d,\; t\P\rbrac{X_i>t} \longrightarrow 1 \quad (\text{as }Y/d \sim \text{unit Pareto})
\end{align*}
Thus, $X\in \R^d_+$ is multivariate regularly varying of index 1, has asymptotically 1-Pareto marginals and has an angular probability measure $\sigma.$

Let $Z=\rbrac{F^{-1}\rbrac{F_{X_i}(X_i)}}_{i=1}^d$, where $F_{X_i}$ denotes the distribution function of $X_i$. Then, $Z_i$ are identically distributed with common distribution function $F$ and $Z$ is multivariate regularly varying with index 1 since the property of regular variation is invariant to increasing monotone marginal transformations. Moreover, due to the fact that $\forall i,\;\bar{F}_{X_i}(x)\sim 1/x \sim \bar{F}(x)$ as $x \to \infty$,  the exponent measure of $Z,\;\mu_Z$ is same as that of $X,\; \mu_X$. As a result, the angular measure of $Z,\;\sigma_Z=\sigma$.

Due to the assumption in this direction that the $(F,h)$-combination test is universally calibrated, we have that 
\begin{align}\label{e: supp_h_univ_calib}
    & \lim_{t\to\infty}\frac{\P\rbrac{h(Z)>t}}{\P(Z_1>t)}=1\notag \\
    \implies & \lim_{t \to \infty} t\P\rbrac{h(Z)>t}=1
\end{align}
Now for every 1-homogeneous continuous function, by \cref{l:homogeneous} we know that 
$$
t\Pb[h(Z)>t] \to c_{\mu_Z} \E\{ h(\Theta)\}=d\E\{h(\Theta)\},\ \ \ t\to\infty,
$$
where $\Theta = (\Theta_i)_{i=1}^d$
is a random vector with probability distribution $\sigma$ on the unit simplex $\Delta^{d-1}.$ Combining the above with \eqref{e: supp_h_univ_calib}, 
$$d\E_{\Theta \sim \sigma}[h(\Theta)]=1.$$
As the choice of $\sigma$ which follows \eqref{e: supp_marginal_cond_sigma} was arbitrary, the above equation holds for all such $\sigma$, proving this direction.

\commb{\textbf{If direction: }}Let $X=\rbrac{F^{-1}(1-P_i)}_{i=1}^d$ be multivariate regularly varying of index $1$ and exponent measure $\mu_X$. Suppose $\sigma_X$ is the corresponding angular measure on $\Delta^{d-1}.$ Then,
\begin{equation}\label{e: supp_marginal_exp}
    \E_{\Theta\sim\sigma_X}[\Theta_i] = 1/d,\; \forall i =1,\ldots,d.
\end{equation}
Indeed, $\E(\Theta_1) + \cdots + \E(\Theta_d)=\E( \|\Theta\|_1) =1$ and \cref{p:general} implies $\E\sqbrac{(\Theta_i)_+}=\E\sqbrac{\Theta_i}$ is a positive constant for all $i$.
This means that
$$
t \Pb[ X_i>t] \sim c_{\mu_X}\cdot (1/d) = 1 \implies \ \ c_{\mu_X} = d.
$$
Now, due to the assumption of this direction, we have that $d\E_{\Theta\sim\sigma_X}\sqbrac{h\rbrac{\Theta}}=1.$ Thus, by \cref{l:homogeneous},
\begin{equation*}
    \lim_{t\to\infty}t\P\rbrac{h(X)>t}=c_{\mu_X}\E_\sigma\sqbrac{h\rbrac{\Theta}}=d\E_{\Theta\sim\sigma_X}\sqbrac{h\rbrac{\Theta}}=1.
\end{equation*}
Since $\mu_X$ was chosen arbitrarily, the $(F,h)$-combination test is universally calibrated. This proves the if direction of the claim as well.
\end{proof}

\subsection{Proof of \cref{thm:characterization}} \label{sec:char}

We first prove an auxiliary lemma. 
\begin{lemma}\label{l:AD}
Suppose ${\cal G}=\{g_1,\ldots,g_d\} \subset {\mathbb B}_+(S)$ 
satisfies the anti-dominance condition. If for
some weights $w \in \R^{d}$, we have 
\begin{equation}\label{e:l:AD}
h(\cdot) = \sum_{i=1}^d w_i g_i(\cdot) \in {\mathbb B}_+(S),
\end{equation}
then it implies that $w \in \mathbb{R}_+^{d}$. 
\end{lemma}

\begin{proof}[of \cref{l:AD}] Suppose that \eqref{e:l:AD} holds where
$w_{i_0}<0$ for some $i_0\in\{1,\ldots,d\}$.  Then, let 
${\cal I}:=\{ i\, :\, w_i<0\}$ and observe that since
$h$ and the $g_i$'s are all non-negative, then ${\cal I}^c=\{j\, :\, w_j\ge 0\}$ 
is non-empty.  Thus $\emptyset\not={\cal I}\subsetneq \{1,\ldots,d\}$.  On the 
other hand, Relation \eqref{e:l:AD} can be equivalently written as
$$
h(x) =  \sum_{j\in {\cal I}^c} w_j g_j(x) - \sum_{i\in{\cal I}} |w_i| g_i(x),\ \ x\in S.
$$
This, since $h$ is a non-negative function, entails
that
$$
\sum_{i\in{\cal I}} |w_i| g_i(x) \le \sum_{j\in {\cal I}^c} w_j g_j(x),\ \ \forall x\in S,
$$
where $|w_{i_0}|>0$ for some $i_0\in {\cal I}$.  This 
contradicts the anti-dominance condition.
\end{proof}

\begin{remark} \label{r:anti-dominance} While the anti-dominance condition may 
appear to be stringent, in some cases it is very easy to verify.  
Indeed, suppose that 
$$S = \{ (u_i)_{i=1}^d\, :\, u_i\ge 0,\ \sum_{i=1}^d u_i=1\}$$
is the non-negative unit simplex.  Let 
also $g_i(u) = u_i,\ u\in S$ be the coordinate functions.  
Then, clearly for no choice of $\lambda_i\ge 0$, and a non-empty set 
${\cal I} \subsetneq \{1,\ldots,d\}$ such that $\sum_{i\in {\cal I}}
\lambda_i>0$, can we have
$$
\sum_{i\in {\cal I}} \lambda_i u_i \le \sum_{j\in {\cal I}^c}
\lambda_j u_j,\ \ \forall u = (u_i)_{i=1}^d \in S. 
$$
Indeed, this inequality is violated by taking $u_{j_0}\downarrow 0$,
for some $j_0\in {\cal I}^c$ with $\lambda_{j_0}>0$.
\end{remark}

\begin{proof}[of \cref{thm:characterization}]  For simplicity, and  without loss of generality we will assume that $c=1$.
Assume that $h \in {\mathbb B}_+(S)$ is such that $(h,\mu) = 1$ for all
$\mu\in {\cal M}_c({\cal G})$. We will prove part {\em (i)} in two steps.\\

\noindent
{\bf Step 1.} Consider any set $\{y_i,\ i=1,\ldots,m\}$ containing the fixed set of points 
$\{x_1,\ldots,x_d\}$ and define the matrix
$$
D = (g_i(y_j))_{d\times m}.
$$
Notice that $G$ is a sub-matrix of $D$, obtained by selecting the $d$ columns of $D$ that
correspond to the set $\{x_1,\ldots,x_d\}$.

By assumption, we have that $1 := (1,\dots,1)^{\T}$ is an interior point of $G(\R_+^d)$ and hence,
$1$ is also an interior point of $D(\R_+^m)\supset G(\R_+^d)$.

We will show that
\begin{equation}
    \label{e:mu-existence}
    D\mu = 1,\ \ \mbox{ for some  }\mu\in (0,\infty)^m
\end{equation}
that is, the vector $\mu$ has all positive entries.

Let  $\mu_0 = (\mu_0(1),\ldots,\mu_0(m)) \in (0,\infty)^m$ be an arbitrary
vector of strictly positive entries.  Since
$1\in  D(\R_+^m)^\circ$, there exists a sufficiently small $\delta>0$, and a
$\mu_\delta\in \R_+^m$, such that $D\mu_\delta = 1 - \delta D\mu_0$.  Indeed, this follows
from the facts that for all $\epsilon>0$, there exists a $\delta>0$ such that
$1 - \delta D\mu_0 \in B_{1}(\epsilon)$ where
$B_{1}(\epsilon)\subset D(\R_+^m)$.

Now, define
$$
\mu:= \mu_\delta + \delta \mu_0.
$$
Observe that by construction $\mu \in (0,\infty)^m$ has all positive entries and
$$
D\mu =  1 - \delta D\mu_0 +\delta D(\mu_0) = 1.
$$
This completes the proof of \eqref{e:mu-existence}.
We shall use this fact in the following step of the proof.\\

\noindent
{\bf Step 2.} Note that every $\nu\in \R_+^m$ corresponds to a 
measure 
$$\varphi_\nu(du) := \sum_{i=1}^m \nu_i \epsilon_{\{y_i\}}(du),$$
where $\epsilon_{\{y\}}(A) = 1_A(y),\ A\in {\cal S}$ is the unit mass measure at the singleton 
$\{y\}$.  With this correspondence, we have that
$$
(h,\varphi_\nu) = h^\top \nu,
$$
where $h := (h(y_j))_{j=1}^m$. Thus, the assumptions of the theorem entail
$$
h^\top \nu =1,\ \mbox{ for all }\nu\in\R_+^m\mbox{ such that }D\nu = 1
$$
We will show that $h \in V_{{\cal G}}:= {\rm span}(g_i,\ i=1,\ldots,d),$ where
$g_i:= (g_i(y_j))_{j=1}^m$.  Suppose that
$$
h_0:= {\rm Proj}_{V_{{\cal G}}} (h).
$$
Define the vector
$$
\nu_\epsilon:= \mu+ \epsilon (h-h_0),
$$
and notice that since by construction $\mu$ has positive entries, there is an $\epsilon>0$,
such that $\nu_\epsilon \in\R_+^m$.

Then, since $h-h_0 \perp g_i$, we obtain
$D \nu_\epsilon = D\mu = 1$.  This, by assumption implies
$$
h^\top \nu_\epsilon = 1.
$$
Since by assumption we also have $h^\top \mu = 1$, it follows that
$$
0 = h^\top (\nu_\epsilon - \mu) = \epsilon h^\top (h-h_0).
$$
This, however, since $\epsilon>0$, implies that $h-h_0 = 0$.  Indeed,
since $h_0\in V_{\cal G} \perp h-h_0$, it follows that
$$
0 = h^\top (h-h_0) =
(h-h_0)^\top (h-h_0) = \|h-h_0\|^2.
$$
We have thus shown that $h=h_0={\rm Proj}_{V_{\cal G}}(h).$
This means that there exist coefficients $\lambda_i\in \R,\ i=1,\ldots,d$, possibly dependent 
on the set $\{y_j\}$, such that
\begin{equation}\label{e:h(y_j)}
h(y_j) = \sum_{i=1}^d\lambda_i g_i(y_j),\ \ \mbox{ for all } j=1,\ldots,m.
\end{equation}
It remains to show that the coefficients $\lambda_i$ do not depend on the choice of the $\{y_j\}$'s.

Notice, however, that we started with a {\em fixed} set 
$\{x_i,\ i=1,\ldots,d\} \subset\{y_j,\ j=1,\ldots,m\}$,  such that the matrix 
$G = (g_i(x_j))_{d\times d}$ is invertible.  By focusing on a subset of the 
equations in \eqref{e:h(y_j)}, we obtain $\lambda G = \widetilde{h}^\top$, where
$\widetilde{h} = (h(x_i),\ i=1,\ldots,d)$. Hence $\lambda = \widetilde{h}^\top G^{-1}$, which
demonstrates the uniqueness of the vector $\lambda = (\lambda_i,\ i=1,\ldots,d)$.  This
completes the proof of part ({\em i}).\\

Part {\em (ii)} follows from \cref{l:AD} due to the anti-dominance condition.
\end{proof}

\subsection{Proof of \cref{thm:frechet}}\label{sec: fct calib}

\begin{proof}
Result 1 directly follows from the max-stability of the Fr\'echet distribution. 

For result 2, apply \cref{l:homogeneous} with $h=h_{\vee,w}$ -
$$
\lim_{t\to\infty} t\Pb[h_{\vee,w}(X) >t] = \lim_{t \to +\infty} \frac{\Pb(h_{\vee,w}(X) > t)}{\Pb(X_1 >t)}= \frac{c_{\mu}}{\sum_{i=1}^d w_i}\E_{\sigma} \sqbrac{\bigvee_{i=1}^d w_i\Theta_i}
$$
where $\sigma(d\theta)$ is the angular probability measure on $\Delta^{d-1}$ associated with $\mu$, the exponent measure of $X$. With calculations similar to that done in \cref{l:eq_univ_calib}, one can show $c_{\mu}=d$. Now, use the simple bound,

\begin{equation}\label{e: condition for frechet honesty}
    \bigvee_{i=1}^dw_i\Theta_i \le \sum_{i=1}^d w_i\Theta_i
\end{equation}

because $\Theta_i\ge 0,\; \forall i.$ Then,

\begin{equation}\label{e: less than equal 1 frechet}
    \frac{d}{\sum_{i=1}^dw_i}\E_{\sigma}\sqbrac{\bigvee_{i=1}^dw_i\Theta_i} \leq \frac{d}{\sum_{i=1}^dw_i}\sum_{i=1}^dw_i\E_{\sigma}(\Theta_i)=\frac{d}{\sum_{i=1}^dw_i}\sum_{i=1}^dw_i \rbrac{\frac{1}{d}}=1
\end{equation}
Now, the above holds with equality if and only if \eqref{e: condition for frechet honesty} holds with equality $\sigma$-a.s. But,
\begin{equation*}
    \bigvee_{i=1}^dw_i\Theta_i = \sum_{i=1}^d w_i\Theta_i \quad \text{$\sigma$-a.s.} \iff w_iw_j\Theta_i\Theta_j=0\quad \text{$\sigma$-a.s.}, \; \forall i \ne j
\end{equation*}
As we have assumed $w_i>0,\; \forall\; i$, we have,
\begin{align*}
    \bigvee_{i=1}^dw_i\Theta_i = \sum_{i=1}^d w_i\Theta_i \quad \text{$\sigma$-a.s.} & \iff \Theta_i\Theta_j=0\quad \text{$\sigma$-a.s.}, \; \forall i \ne j \\
    &\iff \text{supp}(\sigma) \subseteq\{e_i:i=1,\ldots,d\}
\end{align*}
i.e., exponent measure $\mu$ of $X$ is supported on the (positive) axes only.

Now, for any $1\le i< j\le d$, take $p\in[0,1]$ sufficiently large such that $F_{X_i}^{-1}(p)= F_{X_j}^{-1}(p)>0$. Note that equality between the quantiles holds because both $X_i \text{ and }X_j$ are 1-Fr\'echet. Then,
\begin{align*}
    &\P\rbrac{X_i>F_{X_i}^{-1}(p), X_j>F_{X_j}^{-1}(p)}\\  
    &\le\P\rbrac{X \in \R_+^{i-1} \times \rbrac{F_{X_i}^{-1}(p),\infty} \times \R_+^{j-i-1}\times\rbrac{F_{X_j}^{-1}(p),\infty} \times \R_+^{d-j}}
\end{align*}
Let $t_p=F_{X_i}^{-1}(p)=F_{X_j}^{-1}(p) \implies \lim_{p \to 1-}t_p=\infty$. Thus,
\begin{align*}
    & b(t_p)\P\rbrac{X_i>F_{X_i}^{-1}(p), X_j>F_{X_j}^{-1}(p)} \\
    & \le b(t_p)\P\rbrac{\frac{X}{t_p} \in \R_+^{i-1} \times \rbrac{1,\infty} \times \R_+^{j-i-1}\times\rbrac{1,\infty} \times \R_+^{d-j}}\\
    & \implies \lim_{p \to 1-} b(t_p)\P\rbrac{X_i>F_{X_i}^{-1}(p), X_j>F_{X_j}^{-1}(p)} \\
    & \le \lim_{p \to 1-} b(t_p)\P\rbrac{\frac{X}{t_p} \in \R_+^{i-1} \times \rbrac{1,\infty} \times \R_+^{j-i-1}\times\rbrac{1,\infty} \times \R_+^{d-j}}\\
    & =\mu\rbrac{\R_+^{i-1} \times \rbrac{1,\infty} \times \R_+^{j-i-1}\times\rbrac{1,\infty} \times \R_+^{d-j}}=0
\end{align*}
Now since the $X_i$'s are standard 1-Fr\'echet,
\begin{align*}
     \lim_{t\to\infty}b(t)\P\rbrac{X_j>t}=1 & \implies \lim_{p\to1-}b(t_p)\P\rbrac{X_j>F_{X_j}^{-1}(p)}=1 \text{ or }\\
     b(t_p) & \sim \rbrac{\P\rbrac{X_j>F_{X_j}^{-1}(p)}}^{-1} \text{ as } p\to1-
\end{align*} 
Thus,
\begin{align*}
    & \lim_{p\to1-} b(t_p)\P\rbrac{X_i>F_{X_i}^{-1}(p), X_j>F_{X_j}^{-1}(p)}=0\\
    & \implies \lambda(X_i,X_j)=\lim_{p\to1-} \frac{\P\rbrac{X_i>F_{X_i}^{-1}(p), X_j>F_{X_j}^{-1}(p)}}{\P\rbrac{X_j>F_{X_j}^{-1}(p)}}=0
\end{align*}
i.e., the $X_i$'s are asymptotically independent.\\
 This proves that the support of $\mu$ concentrated on the axes implies $X$ is asymptotically independent. The other direction is proved by \cref{p:same_const_indep_implies_mrv}. Thus, equality holds in \eqref{e: less than equal 1 frechet} iff $X$ is asymptotically independent.
 \end{proof}

\section{Additional numerical results}\label{sec: supplement simulation}
This section contains numerical results that complement those in \cref{sec:num} of the main text.
\cref{fig:calib-autoreg,fig:pow-autoreg} respectively show the type-I error and power of combination tests when the shape matrix of the multivariate $t$-copula is of autoregressive type.
Throughout this section the columns of each figure are indexed by $\nu\in\{0.5,1,10\}$.
In \cref{fig:calib-autoreg} the rows are indexed by $\rho\in\{0.2,0.8\}$ with $d=100$, whereas in \cref{fig:pow-autoreg,fig: relpow_exch_lrt,fig: relpow_autoreg_lrt} the rows are indexed by $d\in\{3,25,100\}$ with $\rho=0.2$.  

\commb{\cref{fig: relpow_exch_lrt,fig: relpow_autoreg_lrt} plot the empirical powers of Pareto, Cauchy, Cauchy+ and Fr\'echet combination tests relative to the oracle likelihood ratio test which is based on the likelihood ratio between $\mathcal{H}_0$ and the simple alternative $\mu = \tau \eta$. Each plot uses either an autoregressive or an exchangeable shape matrix $\Sigma$ with $\rho=0.2$. The likelihood ratio test is calibrated exactly using its distribution under $\mathcal{H}_0$. By construction and the Neyman--Pearson lemma, the power of this likelihood ratio test is an upper bound on the power of any feasible test.} 

\commb{These plots support the findings of \cref{fig:pow-autoreg,fig:pow-exch} when considering power relative to the Pareto combination test, but they also show that compared to the most powerful likelihood ratio test, all the heavy-tailed tests we consider fall well short.}

\begin{figure}[htb]
\centering
\includegraphics[width=0.9\textwidth]{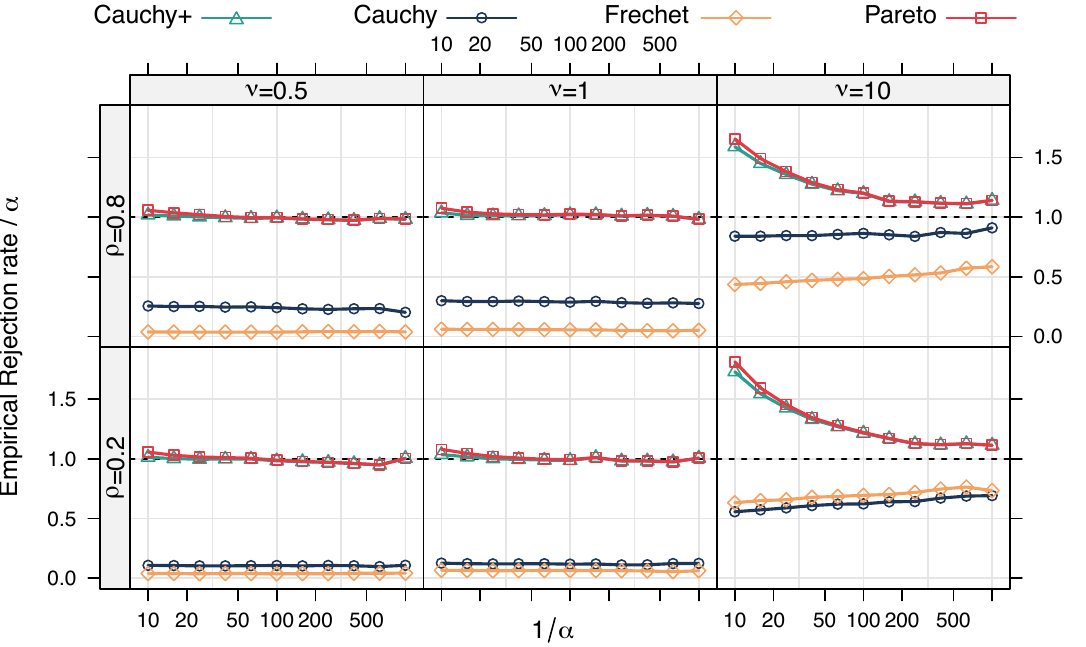}
\caption{Type-I error relative to the nominal level of combination tests under a \commb{100}-dimensional multivariate $t$-copula with $\nu$ degrees of freedom and an autoregressive shape matrix $\Sigma=$
$(\rho^{\abs{i-j}})_{d\times d}$. The curves of Pareto and Cauchy+ overlap almost completely. The results are computed from $5\times10^5$ replications and the standard errors are negligible.}
\label{fig:calib-autoreg}
\end{figure}
\begin{figure}[htb]
\centering
\includegraphics[width=0.9\textwidth]{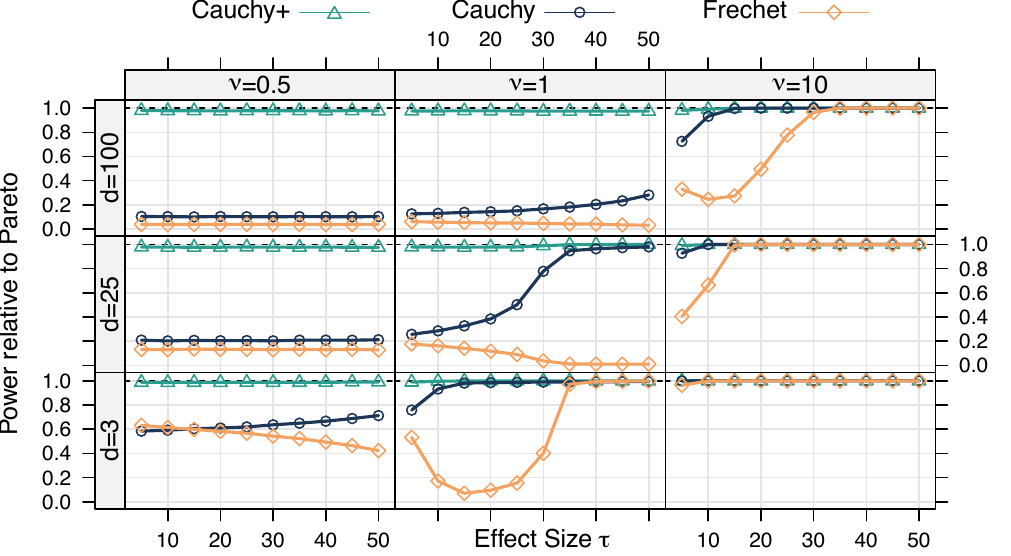}
\caption{\commb{Power of Cauchy, Cauchy+ and Fr\'echet combination tests under $\alpha=0.05$ for testing $\mu=0$ versus $\commb{\mu>0}$ relative to the Pareto combination test. Each combination test is computed from $d$ one-sided $p$-values corresponding to the coordinates of $t_{\nu}(\tau \eta, \Sigma)$, where $\Sigma$ is of autoregressive type with $\rho=0.2$. The curves of Cauchy+ are nearly identical to the $y=1$ line. The results are computed from $5\times10^5$ replications and the standard errors are negligible.}}
\label{fig:pow-autoreg}
\end{figure}
\begin{figure}[htb]
    \centering
    \includegraphics[width=0.9\linewidth]{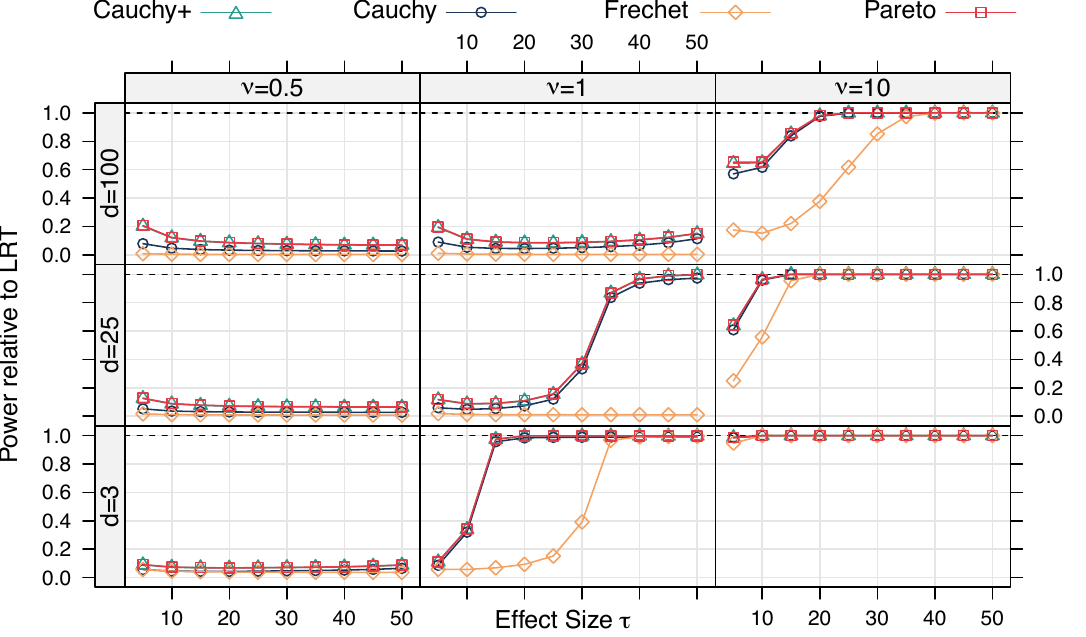}
    \caption{\commb{Power of combination tests under $\alpha=0.05$ for testing $\mu=0$ versus $\mu>0$ relative to the oracle likelihood ratio test. Each combination test is computed from $d$ one-sided $p$-values corresponding to the coordinates of $t_{\nu}(\tau \eta, \Sigma)$, where $\Sigma$ is of exchangeable type with $\rho=0.2$. The curves of Pareto and Cauchy+ overlap almost completely. The results are computed from $5\times10^5$ replications and the standard errors are negligible.}}
    \label{fig: relpow_exch_lrt}
\end{figure}
\begin{figure}[htb]
    \centering
    \includegraphics[width=0.9\linewidth]{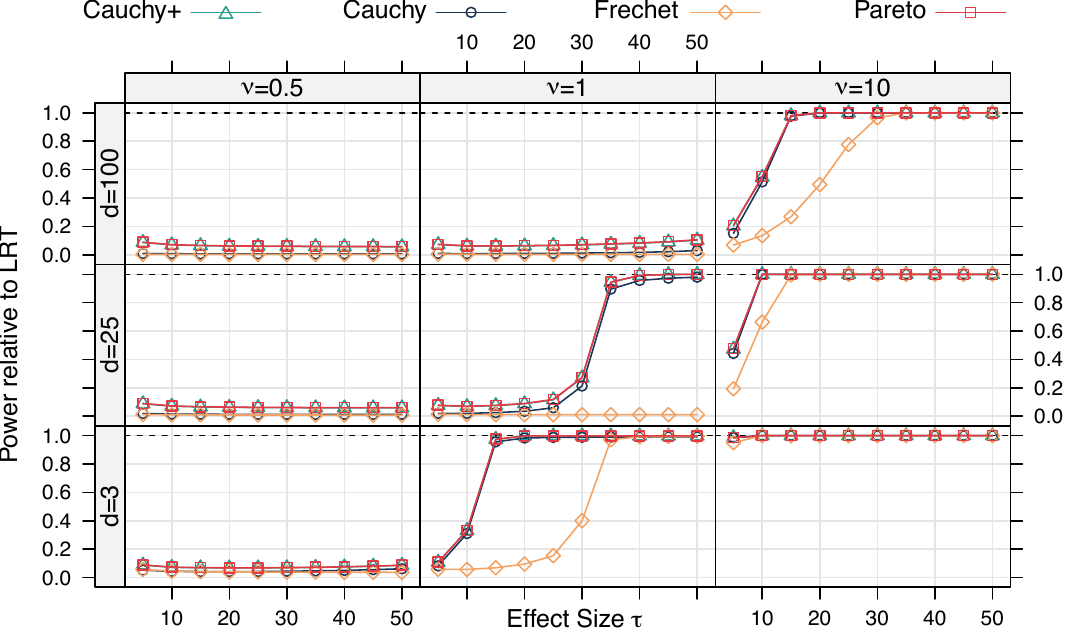}
    \caption{\commb{Power of combination tests under $\alpha=0.05$ for testing $\mu=0$ versus $\mu>0$ relative to the oracle likelihood ratio test. Each combination test is computed from $d$ one-sided $p$-values corresponding to the coordinates of $t_{\nu}(\tau \eta, \Sigma)$, where $\Sigma$ is of autoregressive type with $\rho=0.2$. The curves of Pareto and Cauchy+ overlap almost completely. The results are computed from $5\times10^5$ replications and the standard errors are negligible.}}
    \label{fig: relpow_autoreg_lrt}
\end{figure}

\section{Additional details for application to independence testing with survey data} \label{sec:supp-nhanes}
\begin{table}[htb] \small
\begin{center}
\begin{tabular}{lrrrrrrrrrrr}
& \multicolumn{5}{c}{Female} && \multicolumn{5}{c}{Male}\\\cline{2-6}\cline{8-12}
& $n\;$ & $q_{50}$ & $q_{10}$ & $q_{90}$ & Bonf && $n\;$ & $q_{50}$ & $q_{10}$ & $q_{90}$ & Bonf\\
den/lab &  620 &  0.08 &  0.04 &  0.13 &  0.35 && 648 &  0.01 &  0.01 &  0.03 &  0.04 \\
den/lab &  496 &  0.13 &  0.06 &  0.21 &  0.69 && 519 &  0.05 &  0.02 &  0.11 &  0.19 \\
den/lab &  397 &  0.14 &  0.07 &  0.23 &  0.78 && 415 &  0.07 &  0.03 &  0.14 &  0.28 \\
den/lab &  318 &  0.15 &  0.08 &  0.24 &  0.85 && 332 &  0.08 &  0.03 &  0.15 &  0.36 \\
den/lab &  254 &  0.17 &  0.09 &  0.26 &  1.00 && 266 &  0.10 &  0.04 &  0.19 &  0.50 \\
den/lab &  204 &  0.18 &  0.10 &  0.28 &  1.00 && 213 &  0.12 &  0.06 &  0.22 &  0.64 \\
den/lab &  163 &  0.20 &  0.12 &  0.31 &  1.00 && 170 &  0.15 &  0.07 &  0.25 &  0.90 \\
den/lab &  131 &  0.22 &  0.14 &  0.32 &  1.00 && 136 &  0.19 &  0.10 &  0.29 &  1.00 \\
den/lab &  105 &  0.25 &  0.16 &  0.35 &  1.00 && 109 &  0.22 &  0.12 &  0.32 &  1.00 \\
den/lab &   84 &  0.28 &  0.20 &  0.38 &  1.00 &&  87 &  0.26 &  0.16 &  0.36 &  1.00 \\\hline
bmx/lab &  620 &  0.00 &  0.00 &  0.00 &  0.00 && 648 &  0.00 &  0.00 &  0.00 &  0.00 \\
bmx/lab &  496 &  0.00 &  0.00 &  0.01 &  0.01 && 519 &  0.00 &  0.00 &  0.00 &  0.00 \\
bmx/lab &  397 &  0.01 &  0.00 &  0.02 &  0.02 && 415 &  0.00 &  0.00 &  0.00 &  0.00 \\
bmx/lab &  318 &  0.01 &  0.00 &  0.03 &  0.03 && 332 &  0.00 &  0.00 &  0.00 &  0.00 \\
bmx/lab &  254 &  0.02 &  0.01 &  0.05 &  0.05 && 266 &  0.00 &  0.00 &  0.01 &  0.01 \\
bmx/lab &  204 &  0.03 &  0.01 &  0.07 &  0.11 && 213 &  0.01 &  0.00 &  0.02 &  0.01 \\
bmx/lab &  163 &  0.05 &  0.02 &  0.11 &  0.19 && 170 &  0.01 &  0.00 &  0.03 &  0.04 \\
bmx/lab &  131 &  0.07 &  0.03 &  0.14 &  0.32 && 136 &  0.02 &  0.01 &  0.06 &  0.08 \\
bmx/lab &  105 &  0.11 &  0.06 &  0.19 &  0.61 && 109 &  0.05 &  0.02 &  0.10 &  0.19 \\
bmx/lab &   84 &  0.15 &  0.09 &  0.25 &  1.00 &&  87 &  0.08 &  0.04 &  0.15 &  0.38 \\\hline
dexa/lab &  620 &  0.00 &  0.00 &  0.00 &  0.00 && 648 &  0.00 &  0.00 &  0.00 &  0.00 \\
dexa/lab &  496 &  0.01 &  0.00 &  0.02 &  0.01 && 519 &  0.00 &  0.00 &  0.00 &  0.00 \\
dexa/lab &  397 &  0.01 &  0.00 &  0.02 &  0.02 && 415 &  0.00 &  0.00 &  0.00 &  0.00 \\
dexa/lab &  318 &  0.01 &  0.00 &  0.03 &  0.04 && 332 &  0.00 &  0.00 &  0.01 &  0.01 \\
dexa/lab &  254 &  0.02 &  0.01 &  0.05 &  0.06 && 266 &  0.00 &  0.00 &  0.01 &  0.01 \\
dexa/lab &  204 &  0.03 &  0.01 &  0.07 &  0.11 && 213 &  0.01 &  0.00 &  0.02 &  0.02 \\
dexa/lab &  163 &  0.05 &  0.02 &  0.11 &  0.20 && 170 &  0.01 &  0.01 &  0.04 &  0.05 \\
dexa/lab &  131 &  0.08 &  0.04 &  0.15 &  0.35 && 136 &  0.03 &  0.01 &  0.06 &  0.10 \\
dexa/lab &  105 &  0.11 &  0.06 &  0.20 &  0.64 && 109 &  0.05 &  0.02 &  0.11 &  0.23 \\
dexa/lab &   84 &  0.15 &  0.09 &  0.24 &  1.00 &&  87 &  0.09 &  0.04 &  0.16 &  0.44 \\\hline
\end{tabular}
\caption{Summary statistics for $p$-values testing the null hypothesis of independence between blocks of variables, based on subsamples of the National Health and Nutrition Examination Survey data.}
\label{tab: supp_nhanes}
\end{center}
\end{table}

As noted in \cref{sec: application} and summarized in \cref{nhanes} of the paper, the Pareto combination test yields significant combined $p$-values in five of the six sex $\times$ phenotype settings. The same five settings are also identified using the Bonferroni correction. However, the principal advantage of the Pareto combination test is its substantially greater power at smaller sample sizes, as demonstrated in \cref{tab: supp_nhanes}.

Across each subtable, the Bonferroni combined $p$-values increase much more rapidly with decreasing sample size than those obtained via the Pareto combination test. Focusing on the five sex $\times$ phenotype settings that reject the global null under both methods at the largest sample sizes, we observe that the Pareto combination test rejects the null hypothesis at level $\alpha=0.05$ for all sample sizes at which Bonferroni does so. Moreover, in four of these five settings—{\em bmx/lab} (male and female) and {\em dexa/lab} (male and female)—the Pareto combination test continues to reject the global null for up to 20\% additional sample sizes. When the significance level is relaxed to $\alpha=0.1$, this advantage increases to approximately 30\%. These results demonstrate that the Pareto combination test detects significance in multiple testing scenarios more effectively than the classical Bonferroni correction.

\section{\commb{Additional application to multiple data splitting}}\label{sec: supp_application_data_split}
\commb{In order to test a global null hypothesis when the alternative hypothesis is very large or unspecified, it is of interest to construct an \emph{omnibus test} that has power against a wide range of alternatives. Therefore, it is tempting to construct a test in a \emph{hunt-and-test} fashion: one first learns the specific alternative from which the data appears to have arisen, and then chooses the test statistic accordingly to target that alternative.
Yet, calibrating such a data-adaptive test is often challenging due to the unwieldy dependency between estimating the alternative and assessing its significance. 
To remedy this problem, data splitting has been widely applied: the iid dataset is randomly split into two parts, where one part is first used to choose the test statistic and the other is used to compute the test. 
Such a test can be readily calibrated by ignoring the data-adaptive nature of the test statistic.} 

\commb{Despite the usefulness of such a strategy, as pointed out by \citet{guo2025rank}, data splitting can cause power deficiency and undesired sensitivity to the way that the data is split. Hence, it is worth considering applying the data-splitting test multiple times and combining the $p$-values properly.} 

\commb{Suppose the data-splitting test also depends on a tuning parameter, e.g., the ratio to split data, and for practical purposes it can be chosen from $J$ fixed options. We randomly split the dataset and compute the test statistic $IJ$ times; when the tuning parameter does not affect splitting, it suffices to only split the dataset $I$ times and each time compute the test statistic under every option. 
For $i=1,\dots,I$ and $j=1,\dots,J$, let $P_{ij}$ denote the $p$-value from the $i$th split and the $j$th option. As a straightforward way to combine the $p$-values, one can consider
\[ P_{\min} :=  \min_{i} \min_{j} P_{ij} = \min_{i,j} P_{ij},\]
which takes the minimum among the options for each split, followed by further taking the minimum across the splits. 
For a more general way to combine the $p$-values, let $X_{ij} := -1 / \log(1-P_{ij})$ be the transformed Fr\'echet random variables. Let $w_1,\dots,w_J>0$ with $\sum_{j} w_j = 1$ be some fixed weights assigned to the options of the tuning parameter, e.g., weighting the 1/2 split ratio the most. 
For each split $i$, we first combine $X_{i1},\dots, X_{iJ}$ max-linearly with weights $w$; then we combine the splits by taking their maximum. There is no reason to further weight the splits because they are exchangeable. We have
\[ Y_{i} := \bigvee_{j=1}^{J} w_j X_{ij}, \quad Z := \frac{1}{I}\bigvee_{i=1}^{I} Y_i, \]
which is equivalent to $P_{\min}$ upon choosing $w_1 = \dots = w_J = 1/J$. 
Because $Z$ can be rewritten as 
\[ Z = \bigvee_{i,j} (w_j / I) X_{ij}  = \left. \bigvee_{i,j} (w_j / I) X_{ij} \middle/ \sum_{i,j} (w_j / I) \right. ,\]
we can apply \cref{thm:frechet} and obtain the combined $p$-value
\[ P_{\vee,w} := 1 - \exp(-1/Z). \]
This combined $p$-value is asymptotically honest as the level $\alpha$ approaches zero whenever $X$ is multivariate regularly varying, and strictly conservative unless the $X_{ij}$ are asymptotically independent. As a result one should opt to not use the minimum $p$-value in such a setting.}

\bibliographystyle{imsart-nameyear}

\end{document}